\documentclass[A4paper,reqno,12pt]{amsart}
\usepackage{amssymb}
\usepackage[usenames, dvipsnames]{color}
\usepackage{esint}
\usepackage{hyperref}
\usepackage{todonotes}
\usepackage{verbatim}
\usepackage[margin=1.25in]{geometry}

\usepackage{cite}
\usepackage[nobysame]{amsrefs}
\def\MR#1{}

\theoremstyle{plain}
\newtheorem{theorem}{Theorem}[section]
\newtheorem{lemma}[theorem]{Lemma}
\newtheorem{corollary}[theorem]{Corollary}
\newtheorem{proposition}[theorem]{Proposition}

\theoremstyle{definition}

\theoremstyle{remark}
\newtheorem{remark}[theorem]{Remark}

\newcommand{\dv}{\operatorname{div}}

\newcommand{\dist}{\operatorname{dist}}

\numberwithin{equation}{section}

\newcommand{\bR}{\mathbb{R}}

\newcommand{\bZ}{\mathbb{Z}}
\newcommand{\bS}{\mathbb{S}}

\makeatletter
\def\dashint{\operatorname%
{\,\,\text{\bf--}\kern-.98em\DOTSI\intop\ilimits@\!\!}}
\makeatother

\begin{document}

\title[Optimal estimates for the insulated conductivity problem]{Optimal gradient estimates for the insulated conductivity problem with general convex inclusions case}
\author[H.G. Li]{Haigang Li}
\address[H.G. Li]{School of Mathematical Sciences, Beijing Normal University, Laboratory of MathematiCs and Complex Systems, Ministry of Education, Beijing 100875, China.}
\email{hgli@bnu.edu.cn}

\author[Y. Zhao]{Yan Zhao}
\address[Y. Zhao]{School of Mathematical Sciences, Beijing Normal University, Laboratory of MathematiCs and Complex Systems, Ministry of Education, Beijing 100875, China.}
\email{zhaoyan\_9926@mail.bnu.edu.cn}

\begin{abstract}
We study the insulated conductivity problem which involves two adjacent convex insulators embedded in a bounded domain. It is known that the gradient of solutions may blow up as the distance between the two inclusions tends to zero. However, the sharpness of the blow up rate for general convex insulator case in dimension $n\geq3$ has remained open. The novelty of this paper is that we
answer this problem affirmatively by establishing a pointwise upper bound of the
gradient for general convex insulators, along with a corresponding lower bound that achieves optimal blow up rates. These rates are associated with the first nonzero eigenvalue of an elliptic operator determined by the geometry of insulators. Our results improve and make complete the previous result for ball insulators case studied in \cite{DLY}.
 
\end{abstract}
\maketitle

\section{introduction and main theorem }

In this paper we establish the optimal gradient estimates for the insulated conductivity problem involving two adjacent general convex inclusions. The study of this area originated from \cite{BASL}, where the problem with inclusions closely located in a linear elastic background medium was studied numerically. We consider a bounded domain $D\subset\mathbb{R}^{n}$ $(n\geq2)$ with a $C^{2}$ boundary that  contains two $C^{2,\gamma}$ convex open sets $D_1$ and $D_2$. We assume that they are located far away from the boundary $\partial{D}$, with $\dist(D_1\cup D_2,\partial{D})>c>0$, and that the distance between two inclusions $\varepsilon=\dist(D_1, D_2)$ is small. Let $u$ be the solution to the insulated conductivity problem:
\begin{equation} \label{equkzero}
\left\{
\begin{aligned}
\Delta u&=0\quad \mbox{in}~\widetilde\Omega:=  D\backslash (D_1\cup D_2),\\
\partial_{\nu}u&=0 \quad \mbox{on}~ \partial D_1\cup \partial D_2,\\
u&=\varphi\quad \mbox{on}~ \partial{D},
\end{aligned}
\right.
\end{equation}
for a given function $\varphi\in{C}^{1,\alpha}(\partial{D})$. Here $\partial_{\nu}u$ denotes the normal derivative on boundary $\partial D_1\cup \partial D_2$. It is well-known that there is a unique solution $u$ which is $C^{1,\alpha}$ on $\bar{\widetilde\Omega}$. The solution $u$ represents the electric potential. There is interest from the perspective of engineering in estimating the magnitude of the electric field $\nabla u$ in the narrow region between $D_{1}$ and $D_{2}$. An open problem is to find the optimal upper bound on $|\nabla u|$  in terms of $\varepsilon$. We refer the reader to \cite{ACKLY,BT2,BV,BC,DL,CS,DY,DZ,KLY1,KLY2,KL,Kel} and the references therein for background and related work on this problem.

In the case of $n = 2$, Ammari, Kang, and Lim \cite{AKL} demonstrated an upper bound for $|\nabla u|$ of order $\varepsilon^{-1/2}$ for circular inclusions by using the layer potential method. Ammari et al \cite{AKLLL} further showed that this blow up rate, $\varepsilon^{-1/2}$, is sharp in dimension $n=2$. For insulated conductivity problem \eqref{equkzero}, Bao, Li and Yin \cite{BLY2} also obtained an upper bound of $|\nabla u|$ of order $\varepsilon^{-1/2}$ for all dimensions $n\geq2$ for convex insulators of arbitrary shape. However, the question of whether this upper bound is optimal in dimensions $n\geq3$ remained open for about one decade. It was not until recently that Li and Yang \cite{LY} took advantage of a Harnack inequality to improve the upper bound in dimensions $n\geq3$ to be on the order of $\varepsilon^{-1/2+\beta}$ for some $\beta>0$. Subsequently, Weinkove \cite{Weinkove} employed a direct maximum principle argument and established an upper bound of order $\varepsilon^{-1/2+\beta(n)}$ with a specific constant $\beta(n)>0$ for $n\geq4$ in the case where $D_{1}$ and $D_{2}$ are two balls. Later, Dong, Li, and Yang gave the optimal $\beta(n)$ for a certain class of inclusions, including two balls in all dimension $n\geq3$, particularly with an explicit formula $\beta(n)= [-(n-1)+\sqrt{(n-1)^2+4(n-2)}]/4$ when the insulators are balls in \cite{DLY}. Dong, Yang and Zhu \cite{DYZ} further investigated the insulated conductivity problem with $p$-Laplacian, where the current-electric field relation follows the power law $J=|E|^{p-2}E$.

We use the notation $x=(x',x_n)$ to represent a point in $\mathbb{R}^{n}$, where $x'\in \mathbb{R}^{n-1}$. After a possible translation and rotation of the coordinate if necessary, we assume that the origin $0\in\partial D_2$ is one endpoint of the shortest line between $\partial{D}_{1}$ and $\partial{D}_{2}$, and there exists a universal constant $0<R_0<1$, such that near the origin,  the parts of $\partial D_1$ and $\partial D_2$, denoted by $\Gamma_+$ and $\Gamma_-$, can be represented by graphs of two functions in terms of $x'$. That is, 
\begin{align*}
\Gamma_+ = \left\{ x_n =  \varepsilon+f(x'),~|x'|<2R_0\right\}~~ \mbox{and} ~~\Gamma_- = \left\{ x_n = g(x'),~|x'|<2R_0\right\},
\end{align*}
where $f$ and $g$ are two $C^{2,\gamma}$ functions, $0 < \gamma < 1$, satisfying
\begin{equation}\label{fg_0}
f(x')>g(x')\quad\mbox{for}~~0<|x'|<2R_{0},
\end{equation}
\begin{equation}\label{fg_1}
f(0')=g(0')=0,\quad\nabla_{x'}f(0')=\nabla_{x'}g(0')=0, \quad D^2 (f-g)(0') > 0.
\end{equation}
For $0 < r\leq 2R_{0}$, we denote
\begin{align*}
\Omega_{r}:=\left\{(x',x_{n})\in \Omega~\big|~g(x')<x_{n}< \varepsilon+f(x'),~|x'|<r\right\}.
\end{align*}
Throughout this paper, we denote 
\begin{equation}\label{a_assumption}
a^{ij}=\frac{1}{2}\partial_{ij}(f-g)(0'),\quad\mbox{and}~ a(\xi)=\sum_{i=1}^{n-1}\sum_{j=1}^{n-1}a^{ij}\frac{x_i}{|x'|}\frac{x_j}{|x'|}, \,\,for\,\, \xi\in \bS^{n-2}.
\end{equation} 
It is apparent that $a(\xi)>0\,\, a.e.$ and $\ln a(\xi)\in L^{\infty}(\bS^{n-2}).$ Because $f$ and $g$ are in the class $C^{2,\gamma}$ and satisfy \eqref{fg_0} and \eqref{fg_1}, we  denote, for $|x'|<2R_0$,  
\begin{equation}\label{fg_2}
\begin{aligned}
f(x')-g(x')&=\sum_{i,j=1}^{n-1}a^{ij}x_ix_j+O(|x'|^{2+\gamma}) =a(\xi)|x'|^2+O(|x'|^{2+\gamma}).
\end{aligned}
\end{equation} 

It was proved in \cite{DLY2} that the optimal blow-up rate of the gradient for the insulated conductivity problem \eqref{equkzero} is closely related to the first nonzero eigenvalue of the following eigenvalue problem on $\bS^{n-2}$:  
\begin{equation}\label{SL_problem}
-\dv_{\bS^{n-2}}\Big(a(\xi)\nabla_{\bS^{n-2}}u (\xi)\Big) =\lambda a(\xi) u(\xi), \quad \xi \in \bS^{n-2},
\end{equation}
where $a(\xi)$ is defined by \eqref{a_assumption}, determined by the Hessian $\partial_{ij}(f-g)(0')$.
We define the inner product
\begin{equation}\label{inner_product}
\langle u, v\rangle_{\bS^{n-2}}=\fint_{\bS^{n-2}} a(\xi) uv\,d\sigma.
\end{equation}
From the classical theory, all eigenvalues of \eqref{SL_problem} are real, and the corresponding eigenfunctions  can be normalized to form an orthonormal basis of $L^2(\bS^{n-2})$ under the inner-product defined by \eqref{inner_product}. The first nonzero eigenvalue $\lambda_1$ of problem \eqref{SL_problem} is given by the Rayleigh quotient:
\begin{equation}
\label{first_eigenvalue}
\lambda_1:=\inf_{u\not\equiv 0,\langle u, 1\rangle_{\bS^{n-2}}=0}
\frac{\fint_{\bS^{n-2}}a(\xi)|\nabla_{\bS^{n-2}}u|^2\,d\sigma}
{\fint_{\bS^{n-2}}a(\xi)|u|^2\,d\sigma}.
\end{equation}
Let $\alpha(\lambda_1)$ be the positive root of the quadratic polynomial $\alpha^2 + (n-1)\alpha - \lambda_1$, that is,
\begin{equation}
\label{alpha_lambda_1}
\alpha(\lambda_1) = \frac{-(n-1) + \sqrt{(n-1)^2 + 4 \lambda_1}}{2}.
\end{equation}
According to Lemma 5.1 in \cite{DLY2}, we conclude that $\lambda_1\leq n-2$ and hence $\alpha(\lambda_1)\in(0,1)$. In the case of circular inclusions, where $a(\xi)\equiv a $ is a positive constant and $\lambda_{1}=n-2$, Dong, Li, and Yang in \cite{DLY} proved that the optimal blow-up rate of the gradient is $\varepsilon^{\frac{\alpha(\lambda_1)-1}{2}}$.  In \cite{DLY2}, they extended to study general convex inclusions case and proved that 
\begin{equation}\label{alpha}
|\nabla u(x)|\leq\,C\|u\|_{L^{\infty}(\Omega_{2R_{0}})}(\varepsilon+|x'|^{2})^{\frac{\alpha-1}{2}},\quad\mbox{for}~x\in\Omega_{R_{0}},~\,0\leq\alpha<\alpha(\lambda_1),
\end{equation}
where $C$ is a constant depending on a lower bound of $\alpha(\lambda_1)-\alpha$. However, it is not clear whether $\varepsilon^{\frac{\alpha(\lambda_1)-1}{2}}$ is an optimal upper bound for general convex inclusions. In this paper, we show that $\varepsilon^{\frac{\alpha(\lambda_1)-1}{2}}$ indeed is an upper bound for $\nabla u$, with its sharpness proved in dimension $n=3$ and for a certain class of inclusions in dimensions $n\geq4$.  

By applying standard elliptic estimates, the solution $u\in H^{1}(\widetilde\Omega)$ of \eqref{equkzero} satisfies $\|u\|_{C^{1}(\widetilde\Omega\setminus\Omega_{R_{0}/2})}\leq\,C$. We only need to focus on the following problem near the origin:
\begin{equation}\label{main_problem_narrow}
\left\{
\begin{aligned}
-\Delta u&=0 \quad \mbox{in }\Omega_{2R_0},\\
\partial_{\nu}u&=0 \quad \mbox{on}\,\, \Gamma_+ \cup \Gamma_-.
\end{aligned}
\right.
\end{equation}
Our main results of this paper are as follows.

\begin{theorem}\label{Thm1_upper}
For $n\geq3$, $0<\varepsilon\ll1$, let $u\in H^1(\Omega_{2R_0})$ be a solution of \eqref{main_problem_narrow} with $f$ and $g$ satisfying \eqref{fg_0}, \eqref{fg_1} and \eqref{fg_2}. We have
\begin{equation}\label{mainbound}
|\nabla u(x')|\leq C\|u\|_{L^{\infty}(\Omega_{2R_0})}(\varepsilon+|x'|^2)^{\frac{\alpha(\lambda_1)-1}{2}} \quad \forall x\in\Omega_{R_0 },
\end{equation}
where $\lambda_1$ and $\alpha(\lambda_1)$ are given by \eqref{first_eigenvalue} and \eqref{alpha_lambda_1}, and $C$ is a positive constant depending only on $n$, $R_0$, $\gamma$, $\|f\|_{C^{2,\gamma}(\Gamma^{+})}$ , $\|g\|_{C^{2,\gamma}(\Gamma^{-})}$ and $\|\ln a(\xi)\|_{L^{\infty}(\bS^{n-2})}$.  
\end{theorem}

\begin{remark}
Estimate \eqref{mainbound} improves the exponent from $\frac{\alpha-1}{2}$ in \eqref{alpha} to $\frac{\alpha(\lambda_1)-1}{2}$, which achieves the optimality shown by Theorem \ref{Thm3_lower} below. 
\end{remark}

Applying the maximum principle and Theorem \ref{Thm1_upper}, we  have

\begin{corollary}\label{Cor1}
For $n\geq3$, $0<\varepsilon\ll1$, let $f$ and $g$ satisfy \eqref{fg_0}, \eqref{fg_1} and \eqref{fg_2}. For any solution $u\in H^1(\widetilde\Omega)$ of \eqref{equkzero}, we have
\begin{equation}
\|\nabla u\|_{L^{\infty}(\Omega)}\leq C\|\varphi\|_{C^{2}(\partial{D})}\varepsilon^{\frac{\alpha(\lambda_1)-1}{2}},
\end{equation}
where $\lambda_1$ and $\alpha(\lambda_1)$ are given by \eqref{first_eigenvalue} and \eqref{alpha_lambda_1}, and $C$ is a positive constant depending only on $n$, $R_0$, $\gamma$, $\|\partial D_1\|_{C^{2,\gamma}}$,  $\|\partial D_2\|_{C^{2,\gamma}}$, $\|\partial  \Omega\|_{C^{2}}$ and $\|\ln a(\xi)\|_{L^{\infty}(\bS^{n-2})}$.
\end{corollary} 

The estimate \eqref{mainbound} is shown to be optimal by the following lower bound estimates.

\begin{theorem}\label{Thm3_lower}
For $n\geq3$, $0<\varepsilon\ll 1$, let $D_1$ and $D_2$ be two strictly convex smooth domains in $B_5$, which are symmetric in $x_j$ for each $1\leq j\leq n-1$, with $\lambda_1$ and $\alpha(\lambda_1)$ given by \eqref{first_eigenvalue} and \eqref{alpha_lambda_1}. Assume that the eigenspace corresponding to $\lambda_1$ contains a function which is odd in $x_j$ for some $1\leq j\leq n-1.$ Let $\varphi=x_j$ and $u\in H^1(\widetilde\Omega)$ be the solution of \eqref{equkzero}. Then we have
\begin{equation}\label{eqThm3_lower}
\|\nabla u\|_{L^{\infty}(\tilde \Omega)}\geq\frac{1}{C}\varepsilon^{\frac{\alpha(\lambda_1)-1}{2}},
\end{equation}
where $C$ depends only on $\|\partial D_1\|_{C^4}$, $\|\partial D_2\|_{C^4}$, $\|\ln a(\xi)\|_{L^{\infty}(\bS^{n-2})}$, $\alpha(\lambda_1)$ and $n$.
\end{theorem}

By employing the property of the eigenspace associated with the first nonzero eigenvalue $\lambda_1$ of \eqref{SL_problem}, as proved in \cite{DLY2}(Theorem 5.2 and Corollary 5.7 in \cite{DLY2}), it can be inferred that the conditions outlined in Theorem \ref{lower3} and \ref{lower4} imply the conditions in Theorem \ref{Thm3_lower}. Consequently, the subsequent theorems follow directly.

\begin{theorem}\label{lower3}
For $n = 3$ and $0<\varepsilon\ll1$  , there exist two smooth strictly convex inclusions $D_1, D_2$ inside $D = B_5$ with $D^2(f-g)(0')= M$ for any positive definite matrix $M$, and a boundary data $\varphi \in C^\infty(\partial{D})$ with $\| \varphi \|_{L^\infty(\partial{D})} = 1$, such that the solution $u \in H^1(\widetilde\Omega)$ of \eqref{equkzero} satisfies
$$
 \|\nabla u \|_{L^{\infty}(\widetilde\Omega)} > \frac{1}{C}\varepsilon^{\frac{\alpha(\lambda_1)-1}{2}},
$$
where $\lambda_1$ and $\alpha(\lambda_1)$ are given by \eqref{first_eigenvalue} and \eqref{alpha_lambda_1} with $a(\xi)= \xi^t M \xi$, and $C$ depends only on $\|\partial D_1\|_{C^4}$, $\|\partial D_2\|_{C^4}$, $\|\ln a(\xi)\|_{L^{\infty}(\bS^{n-2})}$, $\alpha(\lambda_1)$ and $n$.
\end{theorem}

\begin{theorem}\label{lower4}
For $n \ge 4$, there exists an $\varepsilon_0 = \varepsilon_0(n) \in (0 ,1/2)$ such that for any positive definite matrix $M$ satisfying
$$
(1-\varepsilon_0) \frac{I}{\|I\|} \le \frac{M}{\|M\|} \le (1+\varepsilon_0) \frac{I}{\|I\|},
$$
there exist two smooth strictly convex inclusions $D_1, D_2$ inside $D = B_5$ with $D^2(f-g)(0') = M$, and a boundary data $\varphi \in C^\infty(\partial{D})$ with $\| \varphi \|_{L^\infty(\partial{D})} = 1$, such that the solution $u \in H^1(\widetilde\Omega)$ of \eqref{equkzero} satisfies
$$
 \|\nabla u \|_{L^{\infty}(\widetilde\Omega)} > \frac{1}{C}\varepsilon^{\frac{\alpha(\lambda_1)-1}{2}},
$$
for $0<\varepsilon\ll1$, where $\lambda_1$ and $\alpha(\lambda_1)$ are given by \eqref{first_eigenvalue} and \eqref{alpha_lambda_1} with $a(\xi)= \xi^t M \xi$, and $C$ is a positive constant depending only on $\|\partial D_1\|_{C^4}$, $\|\partial D_2\|_{C^4}$, $\|\ln a(\xi)\|_{L^{\infty}(\bS^{n-2})}$, $\alpha(\lambda_1)$ and $n$.
\end{theorem}

In the above, $\|M\|$ and $\|I\|$ denote the standard norm of the matrices. 

The rest of this paper is organized as follows. In Section 2, we establish some estimates for the degenerate elliptic operator associated with the problem, which are crucial in the proof of Theorems \ref{Thm1_upper} and \ref{Thm3_lower}. Theorems \ref{Thm1_upper} and \ref{Thm3_lower} are showed in Sections \ref{sec3} and \ref{sec4}, respectively.
 
\section{Estimates on degenerate elliptic operators}

In this section, we establish some estimates for elliptic equations with degenerate coefficients, which are necessary for proving Theorem  \ref{Thm1_upper} and \ref{Thm3_lower}. Throughout this section,  we work in the domain $B_R\subset \bR^{n-1}$ for some $R>0$ and $n\geq 3.$

Let $a(\xi)$ be defined by \eqref{a_assumption}. We introduce the following norms: 
$$
\|F\|_{\varepsilon,\sigma,s, B_R}:=\sup_{x'\in B_R}\frac{|F(x')|}{|x'|^{\sigma}\Big(\varepsilon+a(\xi)|x'|^2\Big)^{1-s}},\quad\mbox{for}\,\,\sigma,s\in \mathbb{R},
$$
and
$$
\|u\|'_{C^1(D)}:=\|u\|_{L^{\infty}(D)}+d\|\nabla u\|_{L^{\infty}(D)},~
\|u\|'_{C^{1,\mu}(D)}:=\|u\|'_{C^1(D)}+d^{1+\mu}[\nabla u]_{C^{\mu}(D)},
$$
where $d=\underset{x,y\in D}\sup|x-y|$, $D\subset\bR^{n-1}$ or $\bR^{n}$. For any bounded set $\Omega\subset\bR^{n-1}$, we denote $H^{1}(\Omega,|x'|^2dx')$ as the weighted $H^1$ norms given by:
$$\|f\|_{H^1(\Omega,|x'|^2dx')}:=\Big(\int_{\Omega}|f|^2|x'|^2dx'\Big)^{\frac{1}{2}}+\Big(\int_{\Omega}|\nabla f|^2|x'|^2dx'\Big)^{\frac{1}{2}}.$$
For $\rho>0$, define the average
$$(u)_{\partial B_{\rho}}^a:=\Big(\int_{\partial B_{\rho}}a(\xi)d\sigma\Big)^{-1}\int_{\partial B_{\rho}}a(\xi)u(x')d\sigma.$$

By using harmonic decomposition , we can get following lemma from \cite{DLY2}.
 
\begin{lemma}[\cite{DLY2},Lemma 2.2]\label{lemma_v1_gen}
For $n \ge 3$, let $a(\xi)$ satisfy \eqref{a_assumption}, $\lambda_1$ and $\alpha(\lambda_1)$ be given by \eqref{first_eigenvalue} and \eqref{alpha_lambda_1}, and $v_1 \in H^1(B_{R_0}, |x'|^2 dx')$ satisfy
\begin{equation*}
\label{equation_v1_gen}
\dv\Big(a(\xi)|x'|^2\nabla v_1\Big)=0\quad\text{in}\,\,B_{R_0} \subset \bR^{n-1}.
\end{equation*}
Then $v_1 \in C^\beta(B_{R_0})$, for some $\beta > 0$ depending only on $n$ and $\| \ln a(\xi)\|_{L^\infty(\bS^{n-2})}$. Moreover, for any $0 < \rho < R \le R_0$, we have $v_1(0') = (v_1)_{\partial B_\rho}^a$, and
\begin{align*}
&\Big( \fint_{\partial B_\rho} a(\xi)|v_1(x') - v_1(0')|^2 \, d\sigma \Big)^{1/2}\le \Big(\frac{\rho}{R} \Big)^{\alpha(\lambda_1)} \Big( \fint_{\partial B_R} a(\xi) |v_1(x') - v_1(0')|^2 \, d\sigma \Big)^{1/2}.
\end{align*}
\end{lemma}

\begin{lemma}\label{corov1}
For $n \ge 3$, let $a(\xi)$ satisfy \eqref{a_assumption}, $\lambda_1$ and $\alpha(\lambda_1)$ be given by \eqref{first_eigenvalue} and \eqref{alpha_lambda_1}, and $v_1 \in H^1(B_{R}, |x'|^2 dx')$ satisfy
\begin{equation*}\label{equation_v11_gen}
\left\{
\begin{aligned}
\dv\Big(a(\xi)|x'|^2\nabla v_1\Big)&=0\,\,\quad\text{in}\,\,B_{R} \subset \bR^{n-1}, \\
v_1&=V_1\quad\text{on}\,\,\partial{B}_{R},
\end{aligned}
\right.
\end{equation*}
where $V_1\in C^{1,\mu}(\bar B_{R}\backslash B_{\frac{3}{4}R})$ for some $\mu\in(0,1)$. Then for any $x'\in \bar B_{R}\backslash \{0'\}$, 
$$|v_1(x')-v_1(0')|+|x'||\nabla v_1(x')|\leq C\Big(\frac{|x'|}{R}\Big)^{\alpha(\lambda_1)}\|V_1-(V_1)_{\partial B_{R}}^a\|'_{C^{1,\mu}(B_{R}\backslash B_{\frac{3}{4}R})},$$
where $C$ depends only on $n$, $\|\ln a\|_{L^{\infty}}$ and $\mu$.
\end{lemma} 

\begin{proof}
For $|x_0'|\in(0,\frac{7}{8}R],$ we take a change of variables by defining $y'=\frac{x'}{|x_0'|}$ to map the region $B_{|x_0'|+\frac{1}{8}|x_0'|}\backslash B_{|x_0'|-\frac{1}{8} |x_0'|}$ to $B_{1+\frac{1}{8} }\backslash B_{1-\frac{1}{8} }$. Let $\tilde v_1(y')=v_1(|x_0'|y')=v_1(x')$. Then $\tilde v_1(y')$ satisfies the equation
$$\dv(a(\xi)|y'|^2\nabla \tilde v_1)=0,\quad y'\in B_{1+\frac{1}{8}}\backslash B_{1-\frac{1}{8}}.$$
By applying the standard elliptic theory and Lemma \ref{lemma_v1_gen}, we obtain
\begin{align*}
&\|\tilde v_1-\tilde v_1(0')\|_{C^1(B_{1+\frac{1}{16}}\backslash B_{1-\frac{1}{16}})}\leq\, C \|\tilde v_1-\tilde v_1(0')\|_{L^2(B_{1+\frac{1}{8}}\backslash B_{1-\frac{1}{8}})}\nonumber\\
\leq&\, 	C \Big(\fint_{B_{|x_0'|+\frac{1}{8}|x_0'|}\backslash B_{|x_0'|-\frac{1}{8}|x_0'|}}|v_1(x')-v_1(0')|^2dx'\Big)^{1/2}\nonumber\\
\leq&\,C \Big(\frac{|x_0'|}{R}\Big)^{\alpha(\lambda_1)}\Big(\fint_{\partial B_{R}}|v_1(x')-v_1(0')|^2dx'\Big)^{1/2}\nonumber\\
\leq&\, C\Big(\frac{|x_0'|}{R}\Big)^{\alpha(\lambda_1)}\|V_1-(V_1)_{\partial B_{R}}^a\|'_{C^{1,\mu}(B_{R}\backslash B_{\frac{3}{4}R})}.
\end{align*}
By rescaling, we have, for $|x_0'|\in(0,\frac{7}{8}R]$, 
\begin{equation}\label{estv11}
|v_1(x_0')-v_1(0')|+|x_0'||\nabla v_1(x_0')|\leq C\Big(\frac{ |x_0'|}{R}\Big)^{\alpha(\lambda_1)}\|V_1-(V_1)_{\partial B_{R}}^a\|'_{C^{1,\mu}(B_{R}\backslash B_{\frac{3}{4}R})}.
\end{equation}

For $|x_0'|\in (\frac{7}{8}R,R]$,  we choose a change of variables $y'=\frac{x'}{R}$ to map the domain $B_{R}\backslash B_{\frac{3}{4}R}$ to $B_{1 }\backslash B_{ \frac{3}{4} }$. Set $\tilde v_1(y')=v_1(Ry')=v_1(x')$ and $\tilde V_1(y')=V_1(Ry')=V_1(x')$. Then $\tilde v_1(y')$ satisfies
\begin{equation*}
\left\{
\begin{aligned}
\dv(a(\xi)|y'|^2\nabla \tilde v_1)&=0\quad\quad\quad y'\in B_{1}\backslash B_{ \frac{3}{4}},\\
\tilde v_1(y')&= \tilde V_1(y')\qquad  y'\in\partial B_1.
\end{aligned}
\right.
\end{equation*}
By the boundary estimates and Lemma \ref{lemma_v1_gen},
\begin{align*}
\|\tilde v_1-\tilde v_1(0')\|_{C^1(B_{1}\backslash B_{\frac{7}{8}})}\leq&\, C \Big(\|\tilde v_1-\tilde v_1(0')\|_{L^2(B_{1}\backslash B_{3/4})}+\|\tilde V_1- (\tilde V_1)_{\partial B_{1}}^a\|_{C^{1,\mu}(B_1\backslash B_{3/4})}\Big)\nonumber\\  
\leq&\, C\|V_1-(V_1)_{\partial B_{R}}^a\|'_{C^{1,\mu}(B_{R}\backslash B_{\frac{3}{4}R})}.
\end{align*} 
By rescaling, we obtain, for $|x_0'|\in(\frac{7}{8}R,R]$,
\begin{equation}\label{estv12}
|v_1(x_0')-v_1(0')|+R|\nabla v_1(x_0')|\leq C\|V_1-(V_1)_{\partial B_{R}}^a\|'_{C^{1,\mu}(B_{R}\backslash B_{\frac{3}{4}R})}.
\end{equation}
Combining \eqref{estv11} and \eqref{estv12}, the proof is finished.
\end{proof}

 \begin{lemma}\label{maximumv22}
For $n \ge 3$, let $a(\xi)$ satisfy \eqref{a_assumption}, $\lambda_1$ and $\alpha(\lambda_1)$ be given by \eqref{first_eigenvalue} and \eqref{alpha_lambda_1}, and $v_2 \in H^1(B_{R} )$ satisfy
\begin{equation*}
\label{equation_v2_gen}
\left\{
\begin{aligned}
\dv\Big[\big(\varepsilon+a(\xi)|x'|^2\big)\nabla v_2\Big]&=G(x')\quad\text{in}\,\,B_{R} \subset \bR^{n-1}, \\
v_2&=V_2 \quad\quad\text{on}\,\,\partial B_{R}.
\end{aligned}
\right.
\end{equation*}
Then 
\begin{equation*}
\|v_2\|_{L^{\infty}(B_{R})}\leq C \varepsilon^{\frac{\alpha(\lambda_1)}{2}-1}\|G\|_{\varepsilon,0,2-\frac{\alpha(\lambda_1)}{2},R}+\|V_2\|_{L^{\infty}(\partial B_{R})},
\end{equation*}
where $C$ depends only on $n$, $\|\ln a(\xi)\|_{L^{\infty}(\bS^{n-2})}$ and $\alpha(\lambda_1)$, independent of $\varepsilon$ and $R$. 
\end{lemma}
\begin{proof}
Without loss of generality, we assume that $\|G\|_{\varepsilon,0,2-\frac{\alpha(\lambda_1)}{2},R}\leq1$. By the definition, this implies that $|G(x')|\leq \Big(\varepsilon+a(\xi)|x'|^2\Big)^{\frac{\alpha(\lambda_1)}{2}-1}$. For a fixed $\tilde\alpha\in(\alpha(\lambda_1),1)$ satisfying
$$\Big(\tilde\alpha-\alpha(\lambda_1)\Big)\frac{\sum_{i=1}^{n-1}(a^{ij}x_j)^2}{\varepsilon+a(\xi)|x'|^2}\leq\frac{1}{2}\sum_{i=1}^{n-1} a^{ii},\quad\text{for}~ |x'|\leq R,$$
we consider an auxiliary function $\psi(x')=-\Big(\varepsilon+a(\xi)|x'|^2\Big)^{\frac{\alpha(\lambda_1)-\tilde\alpha}{2}}$. A direct calculation yields
 \begin{equation*}
\begin{aligned}
&L_{\varepsilon}\psi(x'):=\dv\Big[\Big(\varepsilon
+a(\xi)|x'|^2\Big)\nabla\psi(x')\Big]\\
=&\,\Big(\tilde\alpha-\alpha(\lambda_1)\Big)\Big(\varepsilon+a(\xi)|x'|^2\Big)^{\frac{\alpha(\lambda_1)-\tilde \alpha}{2}}(\sum_{i=1}^{n-1} a^{ii})\\
&\hspace{1cm}-\Big(\tilde\alpha-\alpha(\lambda_1)\Big)^2\Big(\varepsilon+a(\xi)|x'|^2\Big)^{ \frac{\alpha(\lambda_1)-\tilde \alpha}{2}-1}\sum_{i=1}^{n-1}(a^{ij}x_j)^2 \\
=&\,\Big(\tilde\alpha-\alpha(\lambda_1)\Big)\Big(\varepsilon+a(\xi)|x'|^2\Big)^{\frac{\alpha(\lambda_1)-\tilde \alpha}{2}}\Big((\sum_{i=1}^{n-1}a^{ii})-\Big(\tilde\alpha-\alpha(\lambda_1)\Big)\frac{\sum_{i=1}^{n-1}(a^{ij}x_j)^2}{\varepsilon+a(\xi)|x'|^2}\Big) \\
\geq&\,  \frac{1}{2}(\sum_{i=1}^{n-1}a^{ii})\Big(\tilde\alpha-\alpha(\lambda_1)\Big)\Big(\varepsilon+a(\xi)|x'|^2\Big)^{\frac{\alpha(\lambda_1)-\tilde \alpha}{2}}:=\frac{1}{C_0}\Big(\varepsilon+a(\xi)|x'|^2\Big)^{\frac{\alpha(\lambda_1)-\tilde \alpha}{2}},
\end{aligned}
\end{equation*}
where $C_0=\frac{2}{\sum_{i=1}^{n-1} a^{ii}(\tilde\alpha-\alpha(\lambda_1))}$. Since $\frac{\tilde\alpha}{2}-1<0$, it follows that
\begin{equation*}
\begin{aligned}
|G(x')|\leq&\, \Big(\varepsilon+a(\xi)|x'|^2\Big)^{\frac{\alpha(\lambda_1)}{2}-1}   = \Big(\varepsilon+a(\xi)|x'|^2\Big)^{\frac{\alpha(\lambda_1)-\tilde\alpha}{2}+\frac{\tilde\alpha}{2}-1} \\
\leq&\, \varepsilon^{\frac{\tilde\alpha}{2}-1}\Big(\varepsilon+a(\xi)|x'|^2\Big)^{\frac{\alpha(\lambda_1)-\tilde\alpha}{2}}\leq C_0\varepsilon^{\frac{\tilde\alpha}{2}-1}L_{\varepsilon}\psi(x').
\end{aligned}
\end{equation*}
Thus, 
\begin{equation}
\label{equation_v22_gen}
\left\{
\begin{aligned}
\large|L_{\varepsilon} v_2(x')\large|=|G(x')|&\leq L_{\varepsilon}\Big(C_0\varepsilon^{\frac{\tilde\alpha}{2}-1}\psi(x')\Big)\quad\text{in}\,\,B_{R} \subset \bR^{n-1}, \\
v_2&=V_2 \quad\quad\quad\quad\quad\quad\quad\quad\text{on}\,\,\partial B_{R}.
\end{aligned}
\right.
\end{equation}
By using the maximum principle, see e.g. Theorem 8.1 in \cite{GT}, and together with $\|\psi\|_{L^{\infty}(B_{R})}\leq \varepsilon^{\frac{\alpha(\lambda_1)-\tilde\alpha}{2}}$, we complete the proof.
 \end{proof}

 \begin{lemma}\label{Linftyv4}
For $n \ge 3$, let $a(\xi)$ satisfy \eqref{a_assumption}, $\lambda_1$ and $\alpha(\lambda_1)$ be given by \eqref{first_eigenvalue} and \eqref{alpha_lambda_1}, $\sigma>1$, $\alpha>0$ and $v_3 \in H^1(B_{R} )$ satisfy
\begin{equation}
\label{equation_v4_gen}
\left\{
\begin{aligned}
\dv\Big[\big(\varepsilon+a(\xi)|x'|^2\big)\nabla v_3\Big]&= \partial_iF^i (x')+G\quad\text{in}\,\,B_{R} \subset \bR^{n-1}, \\
v_3&=0 \quad\quad\quad\quad\text{on}\,\,\partial B_{R}.
\end{aligned}
\right.
\end{equation}
Then

(i)  if $\|F\|_{\varepsilon,\sigma,1,R} <\infty$ and $G =0$, we have
\begin{equation*}
\|v_3\|_{L^{\infty}(B_{R})}\leq C \left\|F \right\|_{\varepsilon,\sigma,1,R} R^{\sigma-1}  ,
\end{equation*}
where $C$ depends only on $n$, $\sigma$ and $\|\ln a(\xi)\|_{L^{\infty}(\bS^{n-2})}$, independent of $\varepsilon$ and $R$; 

(ii) if $\left\|F_{\sqrt\varepsilon} \right\|_{\varepsilon,1,1,R}<\infty$ and $ G =0$, we have  
\begin{equation*}
\|v_3\|_{L^{\infty}(B_{R})}\leq C \left\|F_{\sqrt\varepsilon} \right\|_{\varepsilon,1,1,R} R   ,
\end{equation*}
where $F_{\sqrt\varepsilon}(x')=\frac{F(x')}{\sqrt\varepsilon}$, and $C$ depends only on $n$ and $\|\ln a(\xi)\|_{L^{\infty}(\bS^{n-2})}$, independent of $\varepsilon$ and $R$. 

(iii) if $\left\|F_{\varepsilon} \right\|_{\varepsilon,1,1,R}<\infty$ and $G =0$, we have
\begin{equation*}
\|v_3\|_{L^{\infty}(B_{R})}\leq C \left\|F_{\varepsilon} \right\|_{\varepsilon,1,1,R} R^2  ,
\end{equation*}
where $F_{\varepsilon}(x')=\frac{F(x')}{\varepsilon}$, and $C$ depends only on $n$ and $\|\ln a(\xi)\|_{L^{\infty}(\bS^{n-2})}$, independent of $\varepsilon$ and $R$.
 
(iv) if $F=0$ and $\|G\|_{\varepsilon,\alpha,1,R}<\infty$, we have
\begin{equation*}
\|v_4\|_{L^{\infty}(B_{R})}\leq C \|G\|_{\varepsilon,\alpha,1,R} R^{\alpha}   ,
\end{equation*}
where $F_{\varepsilon}(x')=\frac{F(x')}{\varepsilon}$, and $C$ depends only on $n$ and $\|\ln a(\xi)\|_{L^{\infty}(\bS^{n-2})}$, independent of $\varepsilon$ and $R$.

\end{lemma}
\begin{proof}
(i) First, we prove the case where $R=1$. Then, through a change of variables, we can complete the proof of (i). Without loss of generality, we  assume that $\|F\|_{\varepsilon,\sigma,1,1}\leq1$, which implies that $|F^i_1(x')|\leq |x'|^{\sigma}$.
 
For $p\geq 2$, multiplying $|v_3|^{p-2}v_3$ in equation \eqref{equation_v4_gen} and using the integration by parts, we obtain
\begin{equation*} 
\begin{aligned}
(p-1)&\int_{B_1(0')}\Big(\varepsilon+a(\xi)|x'|^2\Big)|v_3|^{p-2}|\nabla v_3|^2 dx' 
=(p-1)\int_{B_1(0')}F_1^i\partial_iv_3|v_3|^{p-2} dx'.
\end{aligned}
\end{equation*}
Since $|F(x')|\leq |x'|^{\sigma}$, it follows from the H\"older inequality that 
\begin{equation*} \label{moser11}
\begin{aligned}
\int_{B_1}|x'|^2&|v_3|^{p-2}|\nabla v_3|^2dx'\leq\, C\int_{B_1}|x'|^{\sigma}|v_3|^{p-2}|\nabla v_3| dx'\\
\leq&\, \frac{1}{2}\int_{B_1}|x'|^2|v_3|^{p-2}|\nabla v_3|^2dx'+ C\int_{B_1}|x'|^{2\sigma-2}|v_3|^{p-2}dx',
\end{aligned}
\end{equation*}
which implies that
\begin{equation}\label{equleft}
\int_{B_1}|x'|^2|v_3|^{p-2}|\nabla v_3|^2dx'\leq C\int_{B_1}|x'|^{2\sigma-2}|v_3|^{p-2}dx'.
\end{equation}

By virtue of the following version of the Caffarelli-Kohn-Nirenberg inequality (see \cite{CKN}) in $\bR^{n-1}$:
\begin{equation*}
\label{CKN_inequality}
 \|v\|_{L^{\frac{2(n+1)}{n-1}}(B_{1}, |x'|^2dx')} \le C  \|\nabla v\|_{L^{2}(B_{1}, |x'|^2dx')} \quad \forall v\in H_0^1(B_1(0')),
\end{equation*}
we can deduce
\begin{equation}\label{ineq-v4p}
\begin{aligned}
\|v_3\|_{L^{\frac{p(n+1)}{n-1}}(B_1,|x'|^2dx')}^p=&\,\Big\||v_3|^{\frac{p}{2}}\Big\|^2_{L^{\frac{2(n+1)}{n-1} }(B_1,|x'|^2dx')}\\
\leq&\, C\Big\|\nabla|v_3|^{\frac{p}{2}}\Big\|_{L^2(B_1,|x'|^2dx')} \leq C\int_{B_1}|x'|^2|\nabla|v_3|^{\frac{p}{2}}|^2dx'\\
\leq&\, Cp^2\int_{B_1}|x'|^2|v_3|^{p-2}|\nabla v_3|^2dx'.
\end{aligned}
\end{equation}   
On the other hand, for the right hand side of \eqref{equleft}, we choose $\mu>0$ sufficiently small so that 
$\int_{B_1}|x'|^{(\sigma-1)(n+1+2\mu)-(n-1+2\mu)} dx' <\infty$. Consequently, we obtain 
\begin{equation*}
\begin{aligned}
\int_{B_1}|x'|^{2\sigma-2}|v_3|^{p-2}
\leq&\, \Big(\int_{B_1}|x'|^2|v_3|^{(p-2)\frac{n+1+2\mu}{n-1+2\mu}}dx'\Big)^{\frac{n-1+2\mu}{n+1+2\mu}}\\
&\quad\quad\cdot \Big(\int_{B_1}|x'|^{(\sigma-1)(n+1+2\mu)-(n-1+2\mu)} dx'\Big)^{\frac{n-1+2\mu}{n+1+2\mu}}\\
\leq&\, C \Big(\int_{B_1}|x'|^2|v_3|^{(p-2)\frac{n+1+2\mu}{n-1+2\mu}}dx'\Big)^{\frac{n-1+2\mu}{n+1+2\mu}}.
\end{aligned}
\end{equation*}
This, combining \eqref{equleft} and \eqref{ineq-v4p}, leads to
\begin{equation}\label{moser1}
\|v_3\|_{L^{\frac{p(n+1)}{n-1}}(B_1,|x'|^2dx')}^p\leq Cp^2\|v_3\|_{L^{\frac{(p-2)(n+1+2\mu)}{n-1+2\mu}}(B_1,|x'|^2dx')}^{p-2}. 
\end{equation}

For $p=2$, it follows from \eqref{moser1} that 
\begin{equation}\label{moser11}
\|v_3\|_{L^{\frac{2(n+1+2\mu)}{n-1+2\mu}}(B_1,|x'|^2dx')}\leq C.
\end{equation}
For $p>2$, by virtue of the H\"older inequality, \eqref{moser1} implies that
\begin{equation*} 
\begin{aligned}
 \|v_3\|_{L^{\frac{p(n+1)}{n-1}}(B_1,|x'|^2dx')}^p\leq&\, Cp^2\|v_3\|_{L^{\frac{(p-2)(n+1+2\mu)}{n-1+2\mu}}(B_1,|x'|^2dx')}^{p-2}  \\
\leq&\, \Big(Cp^4\|v_3\|_{L^{\frac{p(n+1+2\mu)}{n-1+2\mu}}(B_1,|x'|^2dx')}^{p-2}\Big)^{\frac{p}{p-2}}+\Big(\frac{1}{p^2}\Big)^{p/2} \\
\leq&\, p^C \|v_3\|_{L^{\frac{p(n+1+2\mu)}{n-1+2\mu}}(B_1,|x'|^2dx')}^{p}+\frac{C}{p^p}.\\
\end{aligned}
\end{equation*}
Therefore
\begin{equation}\label{moser2}
\|v_3\|_{L^{\frac{p(n+1)}{n-1}}(B_1,|x'|^2dx')}\leq p^{\frac{C}{p}} \|v_3\|_{L^{\frac{p(n+1+2\mu)}{n-1+2\mu}}(B_1,|x'|^2dx')} +\frac{C}{p}.
\end{equation}
We take $t_0=\frac{2(n+1+2\mu)}{n-1+2\mu}$, $\chi=\frac{n+1}{n-1}\frac{n-1+2\mu}{n+1+2\mu}>1$, $t_{i+1}=t_i\chi=t_0\chi^{i+1}$, and $p=t_0\frac{n-1+2\mu}{n+1+2\mu}\chi^i=2\chi^i$, and use \eqref{moser11}-\eqref{moser2} to obtain
\begin{equation*}
\begin{aligned}
\|v_3\|_{L^{t_{i+1}}(B_1,|x'|^2dx')}\leq&\, (2\chi^i)^{\frac{C}{2\chi^i}}\|v_3\|_{L^{t_{i}}(B_1,|x'|^2dx')}+\frac{C}{\chi^i}\\
\leq&\, (2\chi )^{\frac{Ci}{\chi^i}}\|v_3\|_{L^{t_{i}}(B_1,|x'|^2dx')}+\frac{C}{\chi^i}\\
\leq&\, (2\chi)^{\sum\limits_{k=0}^{i}\frac{Ck}{\chi^k}}\|v_3\|_{L^{t_{0}}(B_1,|x'|^2dx')}+\sum_{k=0}^{i}\frac{C}{\chi^k}\leq C.
\end{aligned}
\end{equation*}
Letting $i\to\infty$,  the case where $R=1$ is finished. For $R>0$, we set $y'=\frac{x'}{R}$. Let $\tilde v_3(y')=v_3(x')$ and $\tilde F(y')=\frac{F(x')}{R}$. Then $\tilde v_3$ satisfies 
\begin{equation*}
\label{equation_v31_gen}
\left\{
\begin{aligned}
\dv\Big[\big(\tilde\varepsilon+a(\xi)|y'|^2\big)\nabla \tilde v_3\Big]&= \partial_i\tilde{F^i} (y')\quad\text{in}\,\,B_{1} \subset \bR^{n-1},\\
\tilde v_3&=0 \quad\quad\quad\quad\text{on}\,\,\partial B_{1},
\end{aligned}
\right.
\end{equation*}
where $\tilde\varepsilon=\frac{\varepsilon}{R^2}.$ 
Therefore, we have
\begin{equation}
\|\tilde v_3 \|_{L^{\infty}(B_1)}\leq C\|\tilde F\|_{\tilde \varepsilon,\sigma,1,1}.
\end{equation}
Then, considering that $\|\tilde v_3 \|_{L^{\infty}(B_1)}=\|v_3\|_{L^{\infty}(B_R)}$ and $\|\tilde F\|_{\tilde \varepsilon,\sigma,1,1}\leq C\|F\|_{ \varepsilon,\sigma,1,R}R^{\sigma-1}$,  we can complete the proof of (i).

(ii) We only prove the case that $R=1$ and $\|F_{\sqrt\varepsilon}\|_{\varepsilon,1,1,1}\leq1$, which means that $|F(x')|\leq \sqrt\varepsilon|x'|$. 
By utilizing the inequalities $\sqrt\varepsilon |x'|\leq C\Big(\varepsilon+a(\xi)|x'|^2\Big)$ and $|F(x')|\leq\sqrt\varepsilon|x'|$, we can derive 
\begin{equation}\label{equleft1}
\int_{B_1}|x'| |v_3|^{p-2}|\nabla v_3|^2dx'\leq C\int_{B_1}|x'||v_3|^{p-2}|\nabla v_3| dx'.
\end{equation}
By means of the H\"older inequality and the following version of the Caffarelli-Kohn-Nirenberg inequality (see \cite{CKN}) in $\bR^{n-1}$:
\begin{equation*}
 \|v\|_{L^{\frac{2n}{n-2}}(B_{1}, |x'|dx')} \le C  \|\nabla v\|_{L^{2}(B_{1}, |x'|dx')} \quad \forall v\in H_0^1(B_1(0')),
\end{equation*}
similarly to the proof of (i), we can deduce from \eqref{equleft1} that
\begin{equation*}\label{moser3}
\|v_3\|_{L^{\frac{np}{n-2}}(B_1,|x'|dx')}^p \leq Cp^2 \|v_3\|_{L^{p}(B_1,|x'|dx')}^{p-2}.
\end{equation*}
By employing an iteration argument again, we can obtain $\|v_3\|_{L^{\infty}(B_1)}\leq C$. The proof of (ii) is finished.

(iii) We only prove the case that $R=1$ and $\|F_{\varepsilon}\|_{\varepsilon,1,1,1}\leq1$, indicating that $|F(x')|\leq  \varepsilon|x'|$. 
By using $ \varepsilon \leq  \varepsilon+a(\xi)|x'|^2 $ and $|F(x')|\leq \varepsilon|x'|$, we have 
\begin{equation}\label{equleft2}
\int_{B_1}  |v_3|^{p-2}|\nabla v_3|^2dx'\leq C\int_{B_1} |v_3|^{p-2}|\nabla v_3| dx'.
\end{equation}
We apply the following Sobolev-Poincar\'{e} inequality in $\bR^{n-1}$:
\begin{equation*}
 \|v\|_{L^{\frac{2n}{n-2}}(B_{1} )} \le C  \|\nabla v\|_{L^{2}(B_{1})} \quad \forall v\in H_0^1(B_1(0')),
\end{equation*}
 Similar to the proof of (i), it follows from \eqref{equleft2} that
\begin{equation*}\label{moser3}
\|v_3\|_{L^{\frac{np}{n-2}}(B_1 )}^p \leq Cp^2 \|v_3\|_{L^{p}(B_1)}^{p-2}.
\end{equation*}
By employing the iteration argument again, we conclude that $\|v_3\|_{L^{\infty}(B_1)}\leq C.$ The proof of (iii) is finished.

(iv) We only prove the case that $R=1$ and $\|G\|_{\varepsilon,\alpha,1,1}\leq1$, indicating that $|G(x')|\leq   |x'|^{\alpha}$. 
By using $ a(\xi)|x'|^2  \leq  \varepsilon+a(\xi)|x'|^2 $ and $|G(x')|\leq |x'|^{\alpha}$, we have 
\begin{equation}\label{equleft10}
\int_{B_1}|x'|^2|v_3|^{p-2}|\nabla v_3|^2dx'\leq C\int_{B_1}|x'|^{ \alpha}|v_3|^{p-2}dx', 
\end{equation}

 Similar to the proof of (i), for small $\mu>$ such that $\int_{B_1}|x'|^{(\frac{\alpha}{2})(n+1+2\mu)-(n-1+2\mu)} dx' <\infty$ , we have
\begin{equation}\label{moser10}
\|v_3\|_{L^{\frac{p(n+1)}{n-1}}(B_1,|x'|^2dx')}^p\leq Cp^2\|v_3\|_{L^{\frac{(p-2)(n+1+2\mu)}{n-1+2\mu}}(B_1,|x'|^2dx')}^{p-2}. 
\end{equation}
We conclude that $\|v_3\|_{L^{\infty}(B_1)}\leq C$ by employing the iteration argument again. The proof of (iv) is finished.
\end{proof}

\section{Proof of Theorem \ref{Thm1_upper}}\label{sec3}

In this section, we are dedicated to proving Theorem \ref{Thm1_upper}. Following a similar argument as in Section 4 of \cite{LZ}, we define
$$\bar{u}(x'):=\fint_{g(x')}^{\varepsilon+f(x')}u(x',x_{n})\,dx_{n},$$
then 
\begin{equation}\label{barv222}
\partial_i(\delta(x')\partial_i\bar u)+\partial_i\tilde F^{i}=0,\quad\mbox{where}~\delta(x')=\varepsilon+f(x')-g(x'),
\end{equation}
and 
$$\tilde F^{i}(x')=\int_{g(x')}^{\varepsilon+f(x')}\Big(\frac{x_n-\varepsilon-f(x')}{\delta(x')}\partial_{i}g(x')-\frac{x_n-g(x')}{\delta(x')}\partial_{i}f(x')\Big)\partial_{n}u(x)\,dx_n.$$

Because $\delta(x')=\varepsilon+a(\xi)|x'|^2+O(|x'|^{2+\gamma}),$ we can rewrite \eqref{barv222} as
\begin{equation}\label{barv3}
\partial_i\Big(\Big(\varepsilon+a(\xi)|x'|^2\Big)\partial_i\bar u\Big)=-\partial_iF^{i},\quad F^i(x')=\tilde F^i(x')+O(|x'|^{2+\gamma})\partial_i\bar u,
\end{equation}
and then divide $F^i=F_1^i+F_2^i,$ where
\begin{equation}\label{F1i}
F_1^i(x'):=\int_{f(x')}^{\varepsilon+ f(x')}\Big(\frac{x_n-\varepsilon-f(x')}{\delta(x')}\partial_{i}g(x')-\frac{x_n-g(x')}{\delta(x')}\partial_{i}f(x')\Big)\partial_{n}u(x)\,dx_n,
\end{equation}
and 
\begin{equation}\label{F2i}
F_2^i(x'):=\int_{g(x')}^{ f(x')}\Big(\frac{x_n-\varepsilon-f(x')}{\delta(x')}\partial_{i}g(x')-\frac{x_n-g(x')}{\delta(x')}\partial_{i}f(x')\Big)\partial_{n}u(x)\,dx_n+O(|x'|^{2+\gamma})\partial_i\bar u(x').
\end{equation}

By employing the  ``flipping argument" used in the proof of Theorem 1.2 in \cite{BLY2} or Theorem 1.1 in \cite{DLY}, we derive the following two lemmas. Since their proofs are quite similar, we will only present the proof of Lemma \ref{lemmaestu22}.    

\begin{lemma}\label{lemmaestu11}
For $n\geq3$, $0<\varepsilon \ll1$, let $u\in H^1(\Omega_{2R_0})$ be a solution of \eqref{main_problem_narrow} with $f$ and $g$ satisfying \eqref{fg_0}, \eqref{fg_1} and \eqref{fg_2}. We have
\begin{equation*}
\|\nabla u\|_{L^{\infty}(\Omega_{\sqrt{\varepsilon}})}\leq C\varepsilon^{-1/2}\|u-\bar u(0')\|_{L^{\infty}(\Omega_{2\sqrt{\varepsilon}})},
\end{equation*}
where $C$ depends only on $n$, $\gamma$, $R_0$, $\|f\|_{C^{2,\gamma}(\Gamma^+)}$, $\|g\|_{C^{2,\gamma}(\Gamma^-)}$ and $\|\ln a(\xi)\|_{L^{\infty}(\bS^{n-2})}.$
\end{lemma}

Denote  
\begin{equation}\label{Qst}
Q_{t,s}(z):=\{y=(y',y_n)\in \bR^n|~|y'-z'|<s, \,|y_n|<t\}, \quad Q_{t,s}=Q_{t,s}(0).
\end{equation}
\begin{lemma}\label{lemmaestu22}
 For $n\geq3$, $0<\varepsilon\ll1$, let $f$ and $g$ satisfy \eqref{fg_0}, \eqref{fg_1} and \eqref{fg_2}. For any solution $u\in H^1(\Omega_{2R_0})$ of \eqref{main_problem_narrow} and $ \sqrt{\varepsilon}\leq \rho\leq R_0$ , we have
\begin{equation*}
\|u-\bar u(0')\|'_{C^1(\Omega_{\rho}\backslash\Omega_{\frac{3}{4}\rho})}+\|\bar u-\bar u(0')\|'_{C^{1,\mu}(B_{\rho}(0')\backslash B_{\frac{3}{4}\rho}(0'))}\leq C \Big(\fint_{\Omega_{2\rho}\backslash\Omega_{\rho/2}}|u-\bar u(0')|^2\Big)^{1/2},
\end{equation*}
where $C$ depends only on $n$, $\gamma$, $R_0$, $\|f\|_{C^{2,\gamma}(\Gamma^+)}$, $\|g\|_{C^{2,\gamma}(\Gamma^-)}$ and $\|\ln a(\xi)\|_{L^{\infty}(\bS^{n-2})}.$
\end{lemma}
\begin{proof}
For $\rho\geq\sqrt\varepsilon$, we make use of the following change of variables
\begin{equation*}
\left\{
\begin{aligned}
y' &= x', \\
y_n &=  2 \rho^2 \Big( \frac{x_n - g(x')}{f(x') - g(x')+\varepsilon} - \frac{1}{2} \Big),
\end{aligned}\right.
\quad \forall (x',x_n) \in \Omega_{2\rho} \setminus \Omega_{\rho/2},
\end{equation*}
to map the domain $\Omega_{2\rho} \setminus \Omega_{\rho/2}$ to $Q_{\rho^2,2\rho} \setminus Q_{\rho^2,\rho/2}$, where $Q_{t,s}$ is defined by \eqref{Qst}. Let $v(y) = u(x)$. Then $v(y)$ satisfies
\begin{equation*}
\left\{
\begin{aligned}
-\partial_i(b^{ij}(y) \partial_j v(y)) &=0 \quad \mbox{in } Q_{\rho^2,2\rho} \setminus Q_{\rho^2,\rho/2},\\
 b^{nj}(y) \partial_j v(y) &= 0 \quad \mbox{on } |y_n| =\rho^2,
\end{aligned}
\right.
\end{equation*}
where $
(b^{ij}(y)) = \frac{(\partial_x y) (\partial_x y)^t}{\det (\partial_x y)}$. Then 
$\bar v(y')=\fint_{-\rho^2}^{\rho^2}v(y',y_n)dy_n=\bar u(x')$, 
and satisfies
\begin{equation}\label{equbarv}
\partial_i(b^{ii}\partial_i\bar v)=\partial_iF^i,\quad\mbox{where}~F^i=-\fint_{-\rho^2}^{\rho^2} b^{in}\partial_nvdy_n.
\end{equation}
It is straightforward to verify that
\begin{equation}\label{bij}
\begin{aligned}
&\quad\quad\quad\quad\quad\quad\frac{1}{C}\leq\|b^{ii}\|_{L^{\infty}(Q_{\rho^2,2\rho}\backslash Q_{\rho^2,\rho/2})}\leq C, \quad 1\leq\,i\leq n,\\
&\|b^{in}\|_{L^{\infty}(Q_{\rho^2,2\rho}\backslash Q_{\rho^2,\rho/2})}\leq C\rho,\quad [b^{in}]_{C^{\mu}(Q_{\rho^2,2\rho}\backslash Q_{\rho^2,\rho/2})}\leq C\rho^{1-2\mu}, \quad 1\leq\,i\leq n-1,\\
&\quad\quad\quad\quad\quad\quad\quad\|\nabla b^{ij}\|_{L^{\infty}(Q_{\rho^2,2\rho}\backslash Q_{\rho^2,\rho/2})}\leq C\rho^{-1}, \quad 1\leq\,i,j\leq n-1.
\end{aligned}
\end{equation}

Let $\tilde{b}^{ij}(z) = b^{ij}(\rho z)$ and  $\tilde{v}(z) = v(\rho z)$. Then $\tilde{v}(z)$ satisfies
\begin{equation*}
\left\{
\begin{aligned}
-\partial_i(\tilde b^{ij}(z) \partial_j \tilde v(z)) &=0 \quad \mbox{in } Q_{\rho,2} \setminus Q_{\rho,1/2},\\
\tilde b^{nj}(z) \partial_j \tilde v(z) &= 0 \quad \mbox{on }  |z_n|=\rho,
\end{aligned}
\right.
\end{equation*}
with $
\frac{I}{C} \le \tilde{b} \le CI$ and $\| \tilde{b} \|_{C^{\mu} (Q_{\rho,2} \setminus Q_{\rho,1/2})} \le C$. We define
$$
S_l:= \left\{z \in \bR^n ~\big|~  1/2 < |z'| < 2,~ (2l-1) \rho < z_n < (2l+1) \rho \right\},
$$
for any integer $l$, and $S_{s,t}^m: = \left\{z \in \bR^n ~\big|~  s < |z'| < t,~ |z_n| < m\right\}$. Note that $Q_{\rho,2} \setminus Q_{\rho,1/2} = S_0$. We take the even extension of $\tilde v$ with respect to $y_n=\rho$ and then take the periodic extension (so that the period is equal to $2\rho$).
More precisely, we define, for any $l \in \bZ$, a new function $\hat{v}$ by setting
$$\hat{v}(z) := \tilde{v}\Big(z', (-1)^l\big(z_n - 2l \rho\big)\Big), \quad \forall z \in S_l.$$
We also define the corresponding coefficients, for $k = 1,2, \cdots, n-1$,
$$\hat{b}^{nk}(z)=\hat{b}^{kn}(z) := (-1)^l\tilde{b}^{nk}\Big(z', (-1)^l\big(z_n - 2l \rho\big)\Big),  \quad \forall z \in S_l,$$
and for other indices,
$$\hat{b}^{ij}(z) := \tilde{b}^{ij}\Big(z', (-1)^l\big(z_n - 2l \rho\big)\Big), \quad \forall z \in S_l.$$
Then $\hat{v}$ and $\hat{b}^{ij}$ are defined in the infinite ring $Q_{2, \infty} \setminus Q_{1/2, \infty}$. In particular, $\hat{v}$ satisfies the equation
$$
\partial_i (\hat{b}^{ij} \partial_j \hat{v}) = 0 \quad \mbox{in}\,\,S_{1/2,2}^2.
$$
By \cite{LN}*{Proposition 4.1} and \cite{LY}*{Lemma 2.1}, we have
$$
\| \hat v-\bar v(0') \|_{C^1(S_{5/8,7/4}^1)} \le C \| \hat v-\bar v(0')\|_{L^2(S_{1/2,2}^2)} ,
$$
Rescaling back to $u$, we have
\begin{equation}\label{est1}
\|u-\bar u(0')\|'_{C^1(\Omega_{\frac{7}{4}\rho}\backslash{\Omega_{\frac{5}{8}\rho}})}\leq C \Big(\fint_{\Omega_{2\rho}\backslash\Omega_{\rho/2}}|u-\bar u(0')|^2\Big)^{1/2}.
\end{equation}

For $\frac{11}{16}\rho\leq|y'|\leq\frac{3}{2}\rho$, let $\tilde{b}^{ij}(z) = b^{ij}(y'+\rho^2z',\rho^2z_n)$ and $\tilde{v}(z) = v(y'+\rho^2z',\rho^2z_n)$. Then $\tilde{v}$ satisfies
\begin{equation*}
\left\{
\begin{aligned}
-\partial_i(\tilde b^{ij}(z) \partial_j \tilde v(z)) &=0 \quad \mbox{in } Q_{1, 1} ,\\
\tilde b^{nj}(z) \partial_j \tilde v(z) &= 0 \quad \mbox{on } |z_n|=1.
\end{aligned}
\right.
\end{equation*}
By the $W^{2,p}$ estimates, for $p>n$,
\begin{equation*}
\begin{aligned} 
\|\nabla \tilde v\|_{C^{\mu}(Q_{\frac{1}{2},1})}\leq&\, C \|\tilde v-(\tilde v)_{Q_{1,1}}\|_{W^{2,p}(Q_{1,1})}\leq C \|\tilde v-(\tilde v)_{Q_{1,1}}\|_{L^{\infty}(Q_{1,1})}\\
\leq&\, C \|\nabla \tilde v\|_{L^{\infty}(Q_{1,1})} \leq C\rho^2 \|\nabla  v\|_{L^{\infty}(Q_{\rho^2,\rho^2}(y))}\\
\leq&\,  C\rho^2 \|\nabla  v\|_{L^{\infty}(Q_{\rho^2,\frac{7}{4}\rho}\backslash Q_{\rho^2,\frac{5}{8}\rho})} \leq C \rho\|u-\bar u(0')\|'_{C^1(\Omega_{\frac{7}{4}\rho}\backslash{\Omega_{\frac{5}{8}\rho}})}.
\end{aligned}
\end{equation*}
By rescaling, we obtain
\begin{equation}\label{cmu}
\rho\|\nabla v\|_{L^{\infty}(Q_{\rho^2,\frac{3}{2}\rho}\backslash Q_{\rho^2,\frac{11}{16}\rho})}+\rho^{1+2\mu} [\nabla v]_{C^{\mu}(Q_{\rho^2,\frac{3}{2}\rho}\backslash Q_{\rho^2,\frac{11}{16}\rho})}\leq C  \|u-\bar u(0')\|'_{C^1(\Omega_{\frac{7}{4}\rho}\backslash{\Omega_{\frac{5}{8}\rho}})}.
\end{equation}
Noting that $\bar v$ satisfies \eqref{equbarv}, and using the scaling technique, \eqref{bij}, \eqref{cmu} and \eqref{est1}, we have
\begin{equation*}
\begin{aligned}
\|\bar v-\bar v(0')\|'_{C^{1,\mu}(B_{\rho}(0')\backslash B_{\frac{3}{4}\rho}(0'))}\leq&\, C  \|u-\bar u(0')\|'_{C^1(\Omega_{\frac{7}{4}\rho}\backslash{\Omega_{\frac{5}{8}\rho}})}\\
\leq&\, C \Big(\fint_{\Omega_{2\rho}\backslash\Omega_{\rho/2}}|u-\bar u(0')|^2dx\Big)^{1/2}.
\end{aligned}
\end{equation*}
Returning to $\bar u$, the proof is completed.
\end{proof}

By Lemma \ref{lemmaestu11} and \ref{lemmaestu22}, the following Corollary is immediate.

\begin{corollary}\label{rubaru}
For $n\geq3$, $0<\varepsilon \ll1$, let $f$ and $g$ satisfy \eqref{fg_0}, \eqref{fg_1} and \eqref{fg_2}. For any solution $u\in H^1(\Omega_{2R_0})$ of \eqref{main_problem_narrow} with $\|u\|_{L^{\infty}(\Omega_{2R_0})}\leq 1$, if $x\in\Omega_{R_0}\backslash \Omega_{\sqrt\varepsilon}$ , we have
\begin{equation}\label{rubaru1}
\begin{aligned}
 &\|u-\bar u(0')\|_{C^1(\Omega_{|x'|})}^{'}+\|\bar u-\bar u(0')\|_{C^{1,\mu}(\Omega_{|x'|}\backslash\Omega_{\frac{3}{4}|x'|})}^{'}\\
&\leq C\Big(|x'|^2+\big(\fint_{B_{2|x'|}(0')\backslash B_{|x'|/2}(0')}|\bar u-\bar u(0')|^2dx'\big)^{1/2}\Big);
\end{aligned}
\end{equation}
if $x\in \Omega_{\sqrt\varepsilon}$ , we have
\begin{equation}\label{rubaru2}
 \|u-\bar u(0')\|_{C^1(\Omega_{\sqrt\varepsilon})}^{'}\leq C\Big(\varepsilon+\big(\fint_{B_{4\sqrt\varepsilon}(0')\backslash B_{\sqrt\varepsilon}(0')}|\bar u-\bar u(0')|^2dx'\big)^{1/2}\Big),
\end{equation}
where $C$ depends only on $n$, $\gamma$, $R_0$, $\|f\|_{C^{2,\gamma}(\Gamma^+)}$, $\|g\|_{C^{2,\gamma}(\Gamma^-)}$ and $\|\ln a(\xi)\|_{L^{\infty}(\bS^{n-2})}.$
\end{corollary}
\begin{proof}
Because $a(\xi)|x'|^2\in C^{\infty}(\bR^{n-1})$ and $\|u\|_{L^{\infty}(\Omega_{2R_0})}\leq1$, according to the classical elliptic theory, we have $\|\nabla u\|_{L^{\infty}(\Omega_{\frac{3}{2}R_0}\backslash\Omega_{\frac{1}{2}R_0})}\leq C.$ 
By virtue of Lemma \ref{lemmaestu11} and \ref{lemmaestu22}, we have $$|\nabla u(x)|\leq C(\varepsilon+|x'|^2)^{-1/2}, \quad\mbox{for}\, x\in\Omega_{R_0}.$$
Since $\partial_{\nu}u=0$ on $\Gamma_{\pm}$,
$$|\partial_n u(x)|\leq C|x'||\nabla_{x'}u(x)|\leq C|x'|(\varepsilon+|x'|^2)^{-1/2}\leq C,\quad\mbox{on}~\Gamma_{\pm}.$$
It is immediate from the harmonicity of $\partial_n u$ that 
\begin{equation}\label{partialnu} 
\|\partial_n u\|_{L^{\infty}(\Omega_{R_0})}\leq C.
\end{equation}
Therefore, $|u(x',x_{n})-\bar u(x')|\leq C(\varepsilon+|x'|^2)$, for $x\in\Omega_{R_0}$.
For $x\in\Omega_{R_0}\backslash \Omega_{\sqrt\varepsilon}$, 
\begin{align*}
&\big(\fint_{\Omega_{2|x'|}\backslash\Omega_{|x'|/2}}|u-\bar u(0')|^2dx\big)^{1/2}\\
\leq&\,\big(\fint_{\Omega_{2|x'|}\backslash\Omega_{|x'|/2}}|u-\bar u|^2dx\big)^{1/2}+ \big(\fint_{\Omega_{2|x'|}\backslash\Omega_{|x'|/2}}|\bar u-\bar u(0')|^2dx\big)^{1/2} \\
\leq&\, C|x'|^2+\big(\fint_{\Omega_{2|x'|}\backslash\Omega_{|x'|/2}}|\bar u-\bar u(0')|^2dx\big)^{1/2},
\end{align*}
which implies that \eqref{rubaru1} holds. By using the maximum principle, $\|u-\bar u(0')\|_{L^{\infty}(\Omega_{2\sqrt\varepsilon})}=\|u-\bar u(0')\|_{L^{\infty}(\partial{\Omega}_{2\sqrt\varepsilon}\backslash\Gamma_{\pm})}$. Then \eqref{rubaru2} follows from Lemma \ref{lemmaestu11} and \eqref{rubaru1}.
\end{proof}

\begin{proposition}\label{Thm1_upper1}
For $R_0\leq 1$, $n\geq3$ and $0<\gamma+t<2$, suppose $\bar v\in C^{1,\mu}(B_{R_0})$ is a solution to
\begin{equation}
\dv\Big(\Big(\varepsilon+a(\xi)|x'|^2\Big)\nabla{\bar v}\Big)=\partial_{i} F_1^{i}+\partial_{i} F_2^{i}, \quad in\, B_{R_0}\subset \bR^{n-1}+G,
\end{equation}
where $F_i$ and $G$, respectively, satisfy 
$$\Big\|\frac{F_1}{\varepsilon}\Big\|_{\varepsilon,1,1,B_{R_0}}<+\infty, \quad \|F_2\|_{\varepsilon,\gamma+t+1,1,B_{R_0}}<+\infty,\quad \|G\|_{\varepsilon,\alpha,1,B_{R_0}} <+\infty ,$$
and $a(\xi)$ is defined by \eqref{a_assumption}.
Then for $0<\rho<R\leq R_0,$ we have
\begin{equation*}
\begin{aligned}
&\Big(\fint_{\partial B_{\rho}}a(\xi)|\bar v-\bar v(0')|^2d\sigma\Big)^{1/2}\leq\, \Big(\frac{\rho}{R}\Big)^{\alpha(\lambda_1)} \Big(\fint_{\partial B_{R}}a(\xi)|\bar v-\bar v(0')|^2d\sigma\Big)^{1/2}+C\|G\|_{\varepsilon,\alpha,1,B_{R_0}}R^{\alpha}\\
&\quad+C(\frac{\sqrt\varepsilon}{R})^{\alpha(\lambda_1)}\|\bar v-\bar v(0')\|'_{C^{1,\mu}(B_{R}\backslash B_{\frac{3}{4}R})}+C\Big(\Big\|\frac{F_1}{\varepsilon}\Big\|_{\varepsilon,1,1,B_{R_0}}+\|F_2\|_{\varepsilon,\gamma+t+1,1,B_{R_0}}\Big)R^{\gamma+t},
\end{aligned}
\end{equation*}
where $\alpha(\lambda_1)$ is defined by \eqref{alpha_lambda_1}, and $C$ depends only on $n$, $\gamma$, $t$, $\alpha(\lambda_1)$ and $\|\ln a(\xi)\|_{L^{\infty}(\bS^{n-2})}$.
\end{proposition}
\begin{proof}
Denote $L_{\varepsilon}v:=\dv[(\varepsilon+a(\xi)|x'|^2)\nabla v]$. We divide $\bar v:= v_1+v_2+v_3+v_4+v_5$ in $B_R$, where $v_1\in H^1(B_{R_0},|x'|^2dx')$ satisfies
\begin{equation*}
\left\{
\begin{aligned}
\dv(a(\xi)|x'|^2\nabla v_1)&=0 \quad \mbox{in}~ B_R,\\
v_1&=\bar v \quad \mbox{on}~ \partial B_R,
\end{aligned}
\right.
\end{equation*}
$v_2$ satisfies 
\begin{equation}\label{equv22}
\left\{
\begin{aligned}
L_{\varepsilon}v_2&=-\varepsilon\Delta v_1(x') \quad \mbox{in}~ B_R,\\
v_2&=0 \quad\quad\quad\quad\quad \mbox{on}~ \partial B_R,
\end{aligned}
\right.
\end{equation} 
$v_3$ satisfies 
\begin{equation}\label{equv22}
\left\{
\begin{aligned}
L_{\varepsilon}v_3&=G \quad \mbox{in}~ B_R,\\
v_3&=0 \quad\quad\quad\quad\quad \mbox{on}~ \partial B_R,
\end{aligned}
\right.
\end{equation} 
and $ v_4,v_5$, respectively, satisfy
\begin{equation*}
\left\{
\begin{aligned}
L_{\varepsilon} v_4&=\dv F_1 \quad \mbox{in}~ B_R,\\
v_4&=0  \quad\quad\quad \mbox{on}~ \partial B_R,
\end{aligned}
\right.
\quad\left\{
\begin{aligned}
L_{\varepsilon} v_5&=\dv F_2 \quad \mbox{in}~ B_R,\\
v_5&=0  \quad\quad\quad \mbox{on}~\partial B_R.
\end{aligned}
\right.
\end{equation*}

For $v_1$, by using Lemma \ref{lemma_v1_gen}, 
\begin{equation}\label{estimatev1} 
\Big( \fint_{\partial B_\rho} a(\xi)|v_1(x') - v_1(0)|^2 \, d\sigma \Big)^{1/2} 
\leq \Big(\frac{\rho}{R} \Big)^{\alpha(\lambda_1)} \Big( \fint_{\partial B_R} a(\xi) |v_1(x') - v_1(0)|^2 \, d\sigma \Big)^{1/2}.
\end{equation}
Using Lemma \ref{corov1}, we have, for $|x'|\leq R$
\begin{equation}\label{estv111}
\begin{aligned}
|v_1(x')-v_1(0')|+|x'||\nabla v_1(x')|\leq&\, C\Big(\frac{|x'|}{R }\Big)^{\alpha(\lambda_1)}\|\bar v-(\bar v)_{\partial B_{R}}^a\|'_{C^{1,\mu}(B_{R}\backslash B_{\frac{3}{4}R})},\\
\end{aligned}
\end{equation}
which implies $v_1\in H^{1}(B_R, dx').$ Similarly, we have $v_i\in H^{1}(B_R)$, for $i=2,3,4$.

Let $\hat v_1(x'):=v_1(x')-v_1(0')$. Note that 
\begin{equation}\label{epsilonv1}
\varepsilon\Delta v_1=\dv(\varepsilon\nabla\hat v_1)=L_{\varepsilon}\Big(\frac{\varepsilon}{\varepsilon+a(\xi)|x'|^2}\hat v_1(x')\Big)+\dv\Big[\frac{\varepsilon\nabla a(\xi)|x'|^2}{\varepsilon+a(\xi)|x'|^2}\hat v_1(x')\Big]. 
\end{equation}
Combining this with \eqref{epsilonv1}, we can rewrite \eqref{equv22} as 
\begin{equation}
\left\{
\begin{aligned}
L_{\varepsilon}\Big(v_2+\frac{\varepsilon}{\varepsilon+a(\xi)|x'|^2} \hat v_1\Big)&=-\dv\Big[\frac{\varepsilon\nabla a(\xi)|x'|^2}{\varepsilon+a(\xi)|x'|^2} \hat v_1\Big] \quad\quad\quad \mbox{in}~ B_R,\\
v_2+\frac{\varepsilon}{\varepsilon+a(\xi)|x'|^2} \hat v_1&=\frac{\varepsilon}{\varepsilon+a(\xi)R^2}\Big(\bar v-(\bar v)_{\partial B_R}^a\Big) \quad \mbox{on}~ \partial B_R.
\end{aligned}
\right.
\end{equation}  
Using \eqref{estv111}, we obtain
\begin{align*}
&\left|\dv\Big[\frac{\varepsilon\nabla a(\xi)|x'|^2}{\varepsilon+a(\xi)|x'|^2} \hat v_1\Big]\right|\leq\, C \frac{\varepsilon|x'|^{\alpha(\lambda_1)}}{\varepsilon+a(\xi)|x'|^2}{R^{-\alpha(\lambda_1)}} {\|\bar v-(\bar v)_{\partial B_{R}}^a\|'_{C^{1,\mu}(B_{R}\backslash B_{\frac{3}{4}R})}}\nonumber\\
&\quad\quad\quad\quad\quad\quad\leq\, C \varepsilon   {R^{-\alpha(\lambda_1)}} {\|\bar v-(\bar v)_{\partial B_{R}}^a\|'_{C^{1,\mu}(B_{R}\backslash B_{\frac{3}{4}R})}}(\varepsilon+a(\xi)|x'|^2)^{\frac{\alpha(\lambda_1)}{2}-1}.
\end{align*}
By utilizing Lemma \ref{maximumv22}, \eqref{estv111}, and the H\"older inequality, 
we can derive
\begin{equation}\label{estv22} 
\|v_2\|_{L^{\infty}(B_R)}\leq C \Big(\frac{\sqrt{\varepsilon}}{R }\Big)^{\alpha(\lambda_1)}\|\bar v-(\bar v)_{\partial B_{R}}^a\|'_{C^{1,\mu}(B_{R}\backslash B_{\frac{3}{4}R})}.
\end{equation}

From Lemma \ref{Linftyv4}, it is immediate that 
\begin{equation}\label{estv55}
\|v_3\|_{L^{\infty}(B_R)}\leq C\Big\|G\Big\|_{\varepsilon, \alpha,1,B_{R_0}}R^2 \leq C\Big\|G\Big\|_{\varepsilon,\alpha,1,B_{R_0}}R^{\alpha},
\end{equation}

\begin{equation}\label{estv33}
\|v_4\|_{L^{\infty}(B_R)}\leq C\Big\|\frac{F_1}{\varepsilon}\Big\|_{\varepsilon,1,1,B_{R_0}}R^2 \leq C\Big\|\frac{F_1}{\varepsilon}\Big\|_{\varepsilon,1,1,B_{R_0}}R^{\gamma+t},
\end{equation}
and
\begin{equation}\label{estv44}
\|v_5\|_{L^{\infty}(B_R)}\leq C\|F_2\|_{\varepsilon,\gamma+t+1,1,B_{R_0}}R^{\gamma+t}.
\end{equation}
Then, by combining \eqref{estimatev1}, \eqref{estv22},\eqref{estv55}, \eqref{estv33} and \eqref{estv44}, we obtain
\begin{align*}
&\Big(\fint_{\partial B_{\rho}}a(\xi)|\bar v-\bar v(0')|^2d\sigma\Big)^{1/2}\\
\leq&\, \Big(\fint_{\partial B_{\rho}}a(\xi)|v_1-v_1(0')|^2d\sigma\Big)^{1/2}+C\sum_{i=2}^5\|v_i\|_{L^{\infty}(B_R)}\\
\leq&\, \Big(\frac{\rho}{R}\Big)^{\alpha(\lambda_1)} \Big(\fint_{\partial B_{R}}a(\xi)\Big|\bar v-(\bar v)_{\partial B_R}^a\Big|^2d\sigma\Big)^{1/2}+C\Big(\frac{\sqrt\varepsilon}{R}\Big)^{\alpha(\lambda_1)}\|\bar v- (\bar v)_{\partial B_R}^a\|'_{C^{1,\mu}(B_{R}\backslash B_{\frac{3}{4}R})}\\
&\quad\quad+C\Big(\Big\|\frac{F_1}{\varepsilon}\Big\|_{\varepsilon,1,1,B_{R_0}}+\|F_2\|_{\varepsilon,\gamma+t+1,1,B_{R_0}}\Big)R^{\gamma+t}+C\|G\|_{\varepsilon,\alpha,1,B_{R_0}}R^{\alpha}.
\end{align*}
By using the fact that $\Big(\fint_{\partial B_{R}}a(\xi)\Big|\bar v-(\bar v)_{\partial B_R}^a\Big|^2d\sigma\Big)^{1/2}\leq\Big(\fint_{\partial B_{R}}a(\xi)|\bar v - \bar v(0')|^2d\sigma\Big)^{1/2}$ and $ \|\bar v- (\bar v)_{\partial B_R}^a\|'_{C^{1,\mu}(B_{R}\backslash B_{\frac{3}{4}R})}\leq C\|\bar v- \bar v(0')\|'_{C^{1,\mu}(B_{R}\backslash B_{\frac{3}{4}R})}$, we complete the proof.
\end{proof}
\begin{proof}[Proof of Theorem \ref{Thm1_upper}] 
Without loss of generality, we can assume that $\|u\|_{L^{\infty}(\Omega_{2R_0})}\leq 1$. Denote 
$$\omega(\rho)=\Big(\fint_{\partial B_{\rho}}a(\xi)|\bar u-\bar u(0')|^2d\sigma\Big)^{1/2}.$$
It has been shown in the proof of Theorem 1.3 in \cite{DLY2} that
\begin{equation*}
\omega(\rho)\leq C(\alpha)(\varepsilon+\rho^2)^{\alpha/2},\quad \forall\,\alpha<\alpha(\lambda_1),
\end{equation*}
where $C(\alpha)$ is a constant depending on $ \alpha(\lambda_1)-\alpha.$
Now we choose a fixed $\tilde\alpha$ such that $\tilde\alpha<\alpha(\lambda_1)$ and $\frac{\gamma}{2}+\tilde\alpha>\alpha(\lambda_1)$.
By Corollary \ref{rubaru}, we obtain
\begin{equation}\label{estlemma1}
\| u-\bar u(0')\|'_{C^1(\Omega_{\sqrt{\varepsilon}})}\leq C {\varepsilon}^{\tilde\alpha/2},
\end{equation}
and 
\begin{equation}\label{estlemma2}
\| u-\bar u(0')\|_{C^1(\Omega_{\rho}\backslash\Omega_{\frac{3}{4}\rho})}'+ \|\bar u-\bar u(0')\|_{C^{1,\mu}(B_{\rho}\backslash B_{\frac{3}{4}\rho})}'\leq C\rho^{\tilde\alpha},\quad\mbox{for}\,\, \sqrt{\varepsilon}\leq \rho\leq R_0.
\end{equation}

Recall that $\bar u$ is the solution to \eqref{barv3}:
\begin{align*}
\partial_i\Big(\Big(\varepsilon+a(\xi)|x'|^2\Big)\partial_i\bar u\Big)&=-\partial_iF^{i}=-\partial_iF_1^{i}-\partial_iF_2^{i},
\end{align*} 
where $F_1^{i}$ and $F_2^{i}$ are defined by \eqref{F1i} and \eqref{F2i}.
By using \eqref{estlemma1}, \eqref{estlemma2} and \eqref{partialnu}, we have
$$|F_1^i(x')|\leq C\varepsilon|x'|, \quad |F_2^i(x')|\leq C|x'|^{\gamma+\tilde\alpha+1}.$$

By applying Proposition \ref{Thm1_upper1}, we obtain, for $0<\rho<R<R_0$,
\begin{equation}\label{estomegaR}
\omega(\rho)\leq\Big(\frac{\rho}{R}\Big)^{\alpha(\lambda_1)}\omega(R)+C\Big(\frac{\sqrt\varepsilon}{R}\Big)^{\alpha(\lambda_1)}\|\bar u-\bar u(0')\|'_{C^{1,\mu}(B_{R}\backslash B_{\frac{3}{4}R})}+CR^{\gamma+\tilde\alpha}.
\end{equation}
Thus, for $\sqrt{\varepsilon}^{\frac{\alpha(\lambda_1)}{\alpha(\lambda_1)+\frac{\gamma}{2}}}\leq R <R_0$,
\begin{equation*}
\omega(\rho)\leq\Big(\frac{\rho}{R}\Big)^{\alpha(\lambda_1)}\omega(R)+CR^{\frac{\gamma}{2}}\|\bar u-\bar u(0')\|'_{C^{1,\mu}(B_{R}\backslash B_{\frac{3}{4}R})}+CR^{\gamma+\tilde\alpha}.
\end{equation*}
Taking advantage of \eqref{estlemma2}, we conclude
\begin{equation*} 
\omega(\rho)\leq\Big(\frac{\rho}{R}\Big)^{\alpha(\lambda_1)}\omega(R)+CR^{\frac{\gamma}{2}+\tilde\alpha},\quad\mbox{for}\,\,  \sqrt{\varepsilon}^{\frac{\alpha(\lambda_1)}{\alpha(\lambda_1)+\frac{\gamma}{2}}}\leq R <R_0.
\end{equation*}
Because $\frac{\gamma}{2}+\tilde\alpha>\alpha(\lambda_1)$, through a standard iteration and considering the fact that $\omega(R_0)\leq C$, we can derive
\begin{equation}\label{estomegarho}
\omega(\rho)\leq C\rho^{\alpha(\lambda_1)}, \quad\mbox{for}\,\, \sqrt{\varepsilon}^{\frac{\alpha(\lambda_1)}{\alpha(\lambda_1)+\frac{\gamma}{2}}}\leq\rho\leq R_0.
\end{equation}
According to Corollary \ref{rubaru}, for $R_{\varepsilon}:=2\sqrt{\varepsilon}^{\frac{\alpha(\lambda_1)}{\alpha(\lambda_1)+\frac{\gamma}{2}}}$ we have
\begin{equation}\label{estRepsilon}
\|\bar u-\bar u(0')\|'_{C^{1,\mu}(B_{R_{\varepsilon}}\backslash B_{\frac{3}{4}R_{\varepsilon}})}\leq C R_{\varepsilon}^{\alpha{(\lambda_1)}}.
\end{equation}

By substituting \eqref{estomegarho} and \eqref{estRepsilon} into \eqref{estomegaR} and setting $R=R_{\varepsilon}$, we can derive
\begin{equation*}
\omega(\rho)\leq C\rho^{\alpha(\lambda_1)}+C\varepsilon^{\frac{\alpha(\lambda_1)}{2}}+C\varepsilon^{\frac{\alpha(\lambda_1)}{2} \frac{\gamma+\tilde\alpha}{\frac{\gamma}{2}+\alpha(\lambda_1)}},\quad\mbox{for}~ 0<\rho<R_{\varepsilon}.
\end{equation*}
Note that $\gamma+\tilde\alpha>\frac{\gamma}{2}+\alpha(\lambda_1)$, so
\begin{equation}\label{estomegarho1}
\omega(\rho)\leq C\rho^{\alpha(\lambda_1)}+C\varepsilon^{\frac{\alpha(\lambda_1)}{2}},\quad\mbox{for}~ 0<\rho<R_{\varepsilon}.
\end{equation}
From \eqref{estomegarho} and \eqref{estomegarho1}, we can deduce
\begin{equation}\label{estuoptimal}
\omega(\rho)\leq C(\varepsilon+\rho^2)^{\frac{\alpha(\lambda_1)}{2}},\quad\mbox{for}~0<\rho<R_0.
\end{equation}
By combining this with Corollary \ref{rubaru}, the proof is completed.
\end{proof}

\section{Proof of Theorem \ref{Thm3_lower}}\label{sec4}

In this section, the constants $C_i$, $i=0,1,2,3,4 $, $C$ and $r_0$ depend only on the norms $\|\partial D_1\|_{C^4}$, $\|\partial D_2\|_{C^4}$, $\|\ln a(\xi)\|_{L^{\infty}(\bS^{n-2})}$, $\alpha(\lambda_1)$ and $n$, but are independent of $\varepsilon$.  

\begin{proof}[Proof of Theorem \ref{Thm3_lower}]
 After a suitable rotation in $\mathbb{R}^{n-1}$, we can assume without loss of generality that 
$$(f-g)(x')=\sum_{i=1}^{n-1}a_ix_i^2+E(x'),$$
where $|E(x')|\leq C|x'|^4$.
Set 
$$\bar u(x')=\fint_{g(x')}^{\varepsilon+f(x')}u(x',x_n)dx_n.$$ Then, similar to Section \ref{sec3}, we have
\begin{equation*}
\dv\Big[\Big(\varepsilon+\sum_{i=1}^{n-1}a_{i}x_{i}^{2}\Big)\nabla \bar u\Big]=-\partial_i F^i(x') \quad in\,B_1\in \bR^{n-1},
\end{equation*}
where $ F^i(x')=\int_{g(x')}^{\varepsilon+f(x')} \Big(\frac{x_n-\varepsilon-f(x')}{\delta(x')}\partial_{i}g(x')-\frac{x_n-g(x')}{\delta(x')}\partial_{i}f(x')\Big)\partial_{n}u(x)\,dx_n+O(|x'|^4)\partial_i\bar{u}(x').$

\textbf{Step 1.} We write $\bar u=u_1+u_2$ in $B_1$, where $u_1$ satisfies 
\begin{equation*}
\left\{
\begin{aligned}
\dv\Big[\Big(\sum_{i=1}^{n-1}a_{i}x_{i}^{2}\Big)\nabla u_1\Big]&=\partial_i F^i_1(x')\quad \mbox{in}~ B_1\subset R^{n-1},\\
u_1&=\bar u ~\quad\quad \mbox{on}~\partial B_1,
\end{aligned}
\right.
\end{equation*} 
where $F^i_1(x')=\int_{g(x')}^{f(x')} \Big(\frac{x_n-\varepsilon-f(x')}{\delta(x')}\partial_{i}g(x')-\frac{x_n-g(x')}{\delta(x')}\partial_{i}f(x')\Big)\partial_{n}u(x)\,dx_n+O(|x'|^4)\partial_i\bar{u}(x')$,
and $u_2$ satisfies
\begin{equation*}
\left\{
\begin{aligned}
\dv\Big[\Big(\varepsilon+\sum_{i=1}^{n-1}a_{i}x_{i}^{2}\Big)\nabla u_2\Big]&=\partial_i F^i_2(x')-\varepsilon\Delta u_1(x')\quad \mbox{in}~ B_1\subset\bR^{n-1},\\
u_2&=0 \quad\quad\quad\quad\quad\quad \mbox{on}~\partial B_1,
\end{aligned}
\right.
\end{equation*} 
where $F^i_2(x')=\int_{f(x')}^{\varepsilon+f(x')}\big(\frac{x_n-\varepsilon-f(x')}{\delta(x')}\partial_{i}g(x')-\frac{x_n-g(x')}{\delta(x')}\partial_{i}f(x')\big)\partial_{n}u(x)\,dx_n.$

By a direct calculation, we obtain
\begin{equation}\label{estF1}
|F_1^i(x')|\leq C_0|x'|^{\alpha(\lambda_1)+2}.
\end{equation}
Using the standard regularity theory, we have $\|\bar u\|_{C^{2}(B_1(0')\backslash B_{3/4}(0'))}\leq C_0.$ According to Proposition 2.1 in \cite{DLY2}, we know
\begin{equation*}
|u_1(x')-u_1(0')|\leq C_0|x'|^{\alpha(\lambda_1)}.
\end{equation*} 
Then, by a similar argument as in Corollary \ref{corov1}, for $|x'|\neq0$, we have
\begin{equation*}
|u_1(x')-u_1(0')|+|x'||\nabla u_1(x')|\leq C_0|x'|^{\alpha(\lambda_1)}.
\end{equation*}

It is obvious that $|F_2^i(x')|\leq C_0\varepsilon|x'|.$ We further decompose $u_2=u_{21}+u_{22}$ in $B_1(0')$ such that $u_{21}\in H_0^{1}(B_1(0'))$ and satisfies
\begin{equation*}
\dv\Big[\Big(\varepsilon+\sum_{i=1}^{n-1}a_{i}x_{i}^{2}\Big)\nabla u_{21}\Big]=\partial_i F^i_2(x')\quad in\, B_1\subset\bR^{n-1}.
\end{equation*}
Since $|F_2^i(x')|\leq C_0\varepsilon|x'|$, which implies $\|\frac{F_2}{\sqrt\varepsilon}\|_{\varepsilon,1,1,1}\leq C_0\sqrt\varepsilon$, by Lemma \ref{Linftyv4}, we have $$\|u_{21}\|_{L^{\infty}(B_1(0'))}\leq C_0\sqrt\varepsilon.$$
Then $u_{22}\in H_0^1(B_1(0'))$ is a solution to
\begin{equation*}
\dv\Big[\Big(\varepsilon+\sum_{i=1}^{n-1}a_{i}x_{i}^{2}\Big)\nabla u_{22}\Big]=-\varepsilon\Delta u_1(x').
\end{equation*}
By using a similar argument to derive \eqref{estv22}, we obtain
$$\|u_{22}\|_{L^{\infty}(B_1(0'))}\leq C_0\varepsilon^{\frac{\alpha(\lambda_1)}{2}}.$$
So that
\begin{equation}\label{equu2}
\|u_2\|_{L^{\infty}(B_1)}\leq \|u_{21}\|_{L^{\infty}(B_1)}+\|u_{22}\|_{L^{\infty}(B_1)} \leq C_0\varepsilon^{\frac{\alpha(\lambda_1)}{2}}.
\end{equation}

We  also have
$$\left|\fint_{\bS^{n-2}} a(\xi)  ( u_2(r,\xi)-u_2( 0'))Y_{1,j}(\xi)\, d\xi\right|\leq C_0\varepsilon^{\frac{\alpha(\lambda_1)}{2}},$$ 
and
$$\Big(\fint_{\partial B_r} a(\xi)  ( u_2(r,\xi)-u_2( 0'))^2\, d\xi\Big)^{1/2}\leq C_0\varepsilon^{\frac{\alpha(\lambda_1)}{2}}\quad \forall\, r\in(0,1),$$
where $Y_{k,i}$ is the normalized eigenfunction corresponding to $k$-th nonzero eigenvalue $\lambda_k$ of the problem \eqref{SL_problem}, and $\{Y_{k,i}\}_{k,i}$ forms an orthonormal basis of $L^2(\bS^{1})$ under the inner product \eqref{inner_product}. By assumption, we can use $Y_{1,j}$ to denote the eigenfunction which is odd in $x_j$. Therefore, $Y_{1,j}$ is the eigenfunction corresponding to the first nonzero eigenvalue of $\lambda_1$ of \eqref{SL_problem} in the half sphere $\bS^{n-2} \cap \{x_j > 0 \}$ with zero Dirichlet boundary condition. This implies that $\lambda_1$ is simple and $Y_{1,j}$ does not change its sign in the half sphere. Without loss of generality, we assume $Y_{1,j}>0$ in $\{x_j >0\}$.

\textbf{step 2.} 
For $u_1$, we decompose it as follows:
\begin{equation*}
 u_1(x')-u_1(0')=U_0(r)Y_0+\sum_{k=1}^\infty \sum_{i=1}^{N(k)} U_{k,i}(r)Y_{k,i}(\xi) \quad x' \in B_1\setminus\{0\},
\end{equation*}
where $U_{k,i}(r) = \fint_{\bS^{1}} a(\xi)  (u_1(r,\xi)-u_1(0')) Y_{k,i}(\xi) \, d\xi$ and $U_{k,i} \in C([0,1)) \cap C^\infty((0,1))$. Then $U_{1,j}$ satisfies $U_{1,j}(0) = 0$ and
\begin{equation*}
\label{U_11_equation}
LU_{1,j}:= r^2U_{1,j}''(r) + nr U_{1,j}'(r) - \lambda_1 U_{1,j}(r) = H(r), \quad 0 < r <1,
\end{equation*}
where
\begin{align*}
H(r) &= \fint_{\bS^{n-2}} Y_{1,j}(\xi) \dv F_1 d\xi = \fint_{\bS^{n-2}}  \Big(\partial_r F_1 + \frac{1}{r} \dv_\xi F_1\Big) Y_{1,j}(\xi) \, d\xi\\
&= \partial_r \Big(\fint_{\bS^{n-2}}   F_1 Y_{1,1}(\xi) \, d\xi \Big) - \frac{1}{r}\fint_{\bS^{n-2}}F_1\dv_{\xi}Y_{1,j}(\xi) d\xi\\
&=: A'(r) + B(r), \quad 0 < r < 1,
\end{align*}
and $A(r),B(r) \in C^1([0,1))$ satisfy, in view of \eqref{estF1}, that
\begin{equation}
\label{AB_bounds}
|A(r)| \le C(n)r^{2+\alpha(\lambda_1)}, \quad |B(r)| \le C(n)r^{1+\alpha(\lambda_1)}, \quad 0 < r < 1.
\end{equation}

\textbf{step 3.}
We will show, for some constant $\tilde C_1\geq C_1>0$, that
\begin{equation}
\label{U_11_formula}
U_{1,j}(r) = \tilde C_1 r^{\alpha(\lambda_1)} + v(r), \quad 0<r<1,
\end{equation}
where $|v(r)| \le C_2r^{1+\alpha(\lambda_1)}$.
We use the method of reduction of order to find a bounded solution $v$ satisfying $Lv = H$ in $(0,1)$, and then show that  $|v(r)| \le C_2r^{1+\alpha(\lambda_1)}$. Note that $h = r^{\alpha(\lambda_1)}$ is a solution of $Lh = 0$. Let $v = hw$ and
$$
w(r) := \int_0^r\frac{1}{s^{n+2\alpha(\lambda_1)}} \int_0^s \tau^{n-2+\alpha(\lambda_1)} H(\tau) \, d\tau ds, \quad 0 < r < 1.
$$
By a direct calculation,
$$
Lv = L(hw) = h r^2w'' + \Big( 2r^2h' +  nrh \Big)w' =  H.
$$
By \eqref{AB_bounds}, we  estimate $|w(r)| \le C_2r$. Hence, $|v(r)| \le C_2r^{1+\alpha(\lambda_1)}$.
Since $U_{1,j} - v$ is bounded and satisfies $L(U_{1,j} - v) = 0$ in $(0,1)$, we have $U_{1,j} = \tilde C_1h + v$ and \eqref{U_11_formula} follows. 

Next, we prove that $\tilde C_1\geq C_1>0$, where $C_1$ depends only on $n$, $\|\partial{D_1}\|_{C^4}$ and $\|\partial D_2\|_{C^4}$, not on $\varepsilon.$ 
Since $D_1$ and $D_2$ are strictly convex and symmetric in $x_1,\ldots,x_{n-1}$, it is easy to show that $\partial_\nu x_j \ge 0$ in $\{x_j \ge 0\}$ and $\partial_\nu x_j \le 0$ in $\{x_j \le 0\}$. Therefore, under the assumptions of Theorem \ref{Thm3_lower}, $x_j$ is a subsolution of \eqref{equkzero} in $\{x_j \ge 0\}$, and is a supersolution of \eqref{equkzero} in $\{x_j \le 0\}$. Hence, $u \ge x_j$ in $\{x_j \ge 0\}$ and $u \le x_j$ in $\{x_j \le 0\}$. Then, $u(0',x_n)=0$ and $|u(x',x_n)|\geq x_j$ for $(x',x_n)\in\Omega$. So, $|\bar{u}(x')| \ge |x_j|$ in $B_1 \subset \bR^{n-1}$ and $\bar u(0')=0$. Since $Y_{1,j}$ has the same sign as $x_j$, we have
\begin{equation}\label{equubar}
\fint_{\bS^{n-2}} a(\xi) \bar u(r,\xi)Y_{1,j}(\xi)=\fint_{\bS^{n-2}} a(\xi) (\bar u(r,\xi)-\bar u(0'))Y_{1,j}(\xi)\geq C_3r,
\end{equation}
where $C_3\geq 0.$
Using \eqref{equu2} and \eqref{equubar}, we have
\begin{equation*}
\begin{aligned}
U_{1,j}(r)& = \fint_{\bS^{n-2}} a(\xi)  (u_1(r,\xi)-u_1(0')) Y_{1,j}(\xi) \, d\xi \\
& = \fint_{\bS^{n-2}} a(\xi)  (\bar u(r,\xi)-\bar u(0'))  Y_{1,j}(\xi)\, d\xi- \fint_{\bS^{n-2}} a(\xi)  ( u_2(r,\xi)-u_2( 0'))Y_{1,j}(\xi)\, d\xi \\
&\geq C_3r-C_0\varepsilon^{\frac{\alpha(\lambda_1)}{2}}.
\end{aligned}
\end{equation*}
By virtue of \eqref{U_11_formula} and $|v(r)|\leq C_2r^{1+\alpha(\lambda_1)}$, we have
\begin{equation}\label{U1j1}
\tilde C_1r^{\alpha(\lambda_1)}\geq -C_0\varepsilon^{\frac{\alpha(\lambda_1)}{2}}-C_2r^{1+\alpha(\lambda_1)}+C_3r.
\end{equation}
We choose a small $r_0$ satisfying
\begin{equation} \label{U1j2}
-C_0\varepsilon^{\frac{\alpha(\lambda_1)}{2}}-C_2r_0^{1+\alpha(\lambda_1)}+C_3r_0\geq \frac{C_3}{2}r_0.
\end{equation}
By \eqref{U1j1} and \eqref{U1j2}, we obtain that $\tilde C_1\geq \frac{C_3}{2}r_0^{1-\alpha(\lambda_1)}$.

\textbf{step 4.}
Firstly, we choose a positive constant $C_4$ such that $\frac{C_1C_4^{\alpha(\lambda_1)}}{4}>C_0$. Without loss of generality, we suppose that 
\begin{equation}\label{C42}
C_1(C_4\sqrt\varepsilon)^{\alpha(\lambda_1) }-C_2(C_4\sqrt\varepsilon)^{1+\alpha(\lambda_1)}\geq\frac{C_1}{2}(C_4\sqrt\varepsilon)^{\alpha(\lambda_1) }.
\end{equation}
We set $r_{\varepsilon}=C_4\sqrt\varepsilon$. \eqref{C42} implies $U_{1,j}(r_{\varepsilon})\geq \frac{C_1}{2}r_{\varepsilon}^{\alpha(\lambda_1)}$, so that
\begin{equation*}
\begin{aligned}
&\Big(\fint_{\partial B_{r_{\varepsilon}}}a(\xi)|\bar u-\bar u(0')|^2d\sigma\Big)^{1/2}\\
&\geq \Big(\fint_{\partial B_{r_{\varepsilon}}}a(\xi)( u_1(r,\xi)-u_1(0'))^2d\sigma\Big)^{1/2}-\Big(\fint_{\partial B_{r_{\varepsilon}}}a(\xi)(u_2(r,\xi)-u_2(0'))^2d\sigma\Big)^{1/2}\\
&\geq U_{1,j}(r_{\varepsilon})-C_0 \varepsilon^{\frac{\alpha(\lambda_1)}{2}}\\
&\geq \frac{C_1}{2}r_{\varepsilon}^{\alpha(\lambda_1)}-C_0 \varepsilon^{\frac{\alpha(\lambda_1)}{2}}= \Big(\frac{C_1C_4^{\alpha(\lambda_1)}}{2}-C_0\Big)\varepsilon^{\frac{\alpha(\lambda_1)}{2}}\geq \frac{C_1C_4^{\alpha(\lambda_1)}}{4}\varepsilon^{\frac{\alpha(\lambda_1)}{2}}.
\end{aligned}
\end{equation*}
Hence, there exists an $\xi_0\in \bS^{n-2}$ such that $|\bar u(r_{\varepsilon},\xi_0)|\geq \frac{1}{C}\varepsilon^{\frac{\alpha(\lambda_1)}{2}}.$ Since $\bar u$ is the average of $u$ in the $x_n$ direction, there exists an $x_n$ such that 
\begin{equation*}
|u(x_{\varepsilon}',x_n)|\geq \frac{1}{C}\varepsilon^{\frac{\alpha(\lambda_1)}{2}}, \quad x_{\varepsilon}'=(r_{\varepsilon},\xi_0).
\end{equation*}
This, together with $u(0)=0$, implies that \eqref{eqThm3_lower} holds. 
\end{proof}

\noindent{\bf Acknowledgements.} The work of H. Li was partially supported by Beijing Natural Science Foundation (No.1242006) and the Fundamental Research Funds for the Central Universities (No.2233200015).\\

\noindent{\bf Conflict of Interest.} The authors declare that they have no conflict of interest.

\end{document}